\documentclass[twoside]{irmaems}
\usepackage{amssymb,amsmath,color,enumerate}
\usepackage{wrapfig,epsfig,graphicx,curves}
\usepackage{makeidx}
\usepackage{hyperref}
\usepackage{latexsym}
\usepackage{pgf,tikz}
\usetikzlibrary{arrows}

\usepackage[applemac]{inputenc} 
\usepackage[mathcal]{eucal}
\usepackage{psfrag}
\usepackage{graphicx}
\usepackage{makeidx} 
\usepackage{hyperref}

\renewcommand\theequation{\arabic{section}.\arabic{equation}}

\theoremstyle{definition} 
 
 \newtheorem{definition}{Definition}[section]
  \newtheorem{definitions}{Definitions}[section]
 \newtheorem{remark}[definition]{Remark}
 \newtheorem{remarks}[definition]{Remarks}
 \newtheorem{example}[definition]{Example}

\theoremstyle{plain}      
 \newtheorem{proposition}[definition]{Proposition}
 \newtheorem{theorem}[definition]{Theorem}
 \newtheorem{corollary}[definition]{Corollary}
 \newtheorem{lemma}[definition]{Lemma}

\newcommand{\dist}{\operatorname{dist}}
\renewcommand{\r}{\mathbb{R}}

\title{From Funk to Hilbert Geometry}  

\makeindex 

\author{Athanase Papadopoulos\thanks{The first author is partially supported by  the French ANR project FINSLER} and Marc Troyanov}

\address{
 Institut de Recherche Math\'ematique Avanc\'ee,
 \\
Universit{\'e} de Strasbourg and CNRS,
\\
7 rue Ren\'e Descartes,
\\
 67084 Strasbourg Cedex, France.
\\
email:\,\tt{papadop@math.unistra.fr}
\\[4pt]
\'Ecole Polytechnique F\'ed\'erale de Lausanne, 
\\ Section de Math\'ematiques (B\^atiment MA) Station 8, 
\\
CH-1015 Lausanne, Switzerland
\\
email:\,\tt{marc.troyanov@epfl.ch}
}

\begin{document}

\maketitle

\begin{abstract} 
We survey some basic geometric properties of the Funk metric of a convex set in $\mathbb{R}^n$. 
In particular, we study its geodesics, its topology, its metric balls, its convexity properties, its perpendicularity theory and its isometries. The Hilbert metric is a symmetrization of the Funk metric, and we show some properties of the Hilbert metric that follow directly from the properties we prove for the Funk metric.
\end{abstract}
  


\tableofcontents

 \section{Introduction}
 
The Funk metric associated to an open convex subset of a Euclidean space is a \emph{weak}\index{metric!weak}\index{weak metric} metric in the sense that it does not satisfy all the axioms of a metric: it is not symmetric, and we shall also allow the distance between two points to be zero; see Chapter 1 in this volume \cite{PT-Minkowski}, where such metrics are introdced. Weak metrics often occur in the calculus of variations and in Finsler geometry and the study of such metrics has been revived recently in low-dimensional topology and geometry by Thurston who introduced an asymmetric metric on Teichm\"uller space,\index{metric!Thurston}\index{Thurston metric} which became the subject of intense research. In this paper, we shall sometimes use the expression ``metric" instead of ``weak metric" in order to simplify.
 
 The Funk weak metric $F_{\Omega}(x,y)$ is associated to an open convex subset $\Omega$ of a Euclidean space $\mathbb{R}^n$. It is important an important metric for the subject treated in this handbook, because
the Hilbert metric $H_{\Omega}$ of the convex set $\Omega$ is the \emph{arithmetic symmetrization}\index{arithmetic symmetrization!metric}\index{metric!arithmetic symmetrization} of its  
 Funk metric. More precisely, for any $x$ and $y$ in $\Omega$, we have
\[
 H_{\Omega}(x,y)= \frac{1}{2}\left( F_{\Omega}(x,y)+ F_{\Omega}(y,x)
\right).
\]

Independently of its relation with the Hilbert metric, the Funk metric is a nice example of a weak metric, and there are many natural questions for a given convex subset $\Omega$ of  $\mathbb{R}^n$ that one can ask and solve for such a metric, regarding its geodesics, its balls, its isometries, its boundary structure, and so on. Of course, the answer depends on the shape of the convex set $\Omega$, and it is an interesting aspect of the theory to study  the influence of the properties of the boundary $\partial \Omega$ (its degree of smoothness, the fact that it is a polyhedron, a strictly convex hypersurface, etc.) on the Funk geometry of this set. The same questions can be asked for the Hilbert geometry, and they are addressed in several chapters of this volume,   although the name ``Funk metric" is not used there, but it is used by Busemann in his later papers and books, see e.g. \cite{Busemann-homothetic}, and in the memoir \cite{Zaustinsky} by Zaustinsky (who was a student of  Busemann). 

We studied some aspects of this metric in \cite{PT2}, following Busemann's ideas. But a systematic study of this metric is something which seems to be still missing in the literature, and the aim of this paper is somehow to fill this gap. 
 
 Euclidean segments in $\Omega$ are geodesics for the Funk metrics, and in the case where the domain $\Omega$ is strictly convex (that is, if there is no nonempty open Euclidean segment in $\partial \Omega$), the Euclidean segments are the unique geodesic segments. For a metric $d$ on a subset of Euclidean space, the property of having the Euclidean segments $d$-geodesics is the subject of Hilbert's Problem IV (see Chapter 15 in this volume  \cite{Pap}). One version of this problem asks for a characterization of non-symmetric metrics on subsets of $\mathbb{R}^n$ for which the Euclidean segments are geodesics. 
 
The Funk metric is also a basic example of a Finsler structure,  perhaps one of the  most basic one. In the paper \cite{PT2}, we introduced the notions of  \emph{tautological}\index{tautological!Finsler structure}\index{Finsler structure!tautological} and of {\it reversible tautological}\index{reversible tautological!Finsler structure}\index{Finsler structure!tautological} Finsler structure of the domain $\Omega$. The Funk metric $F_{\Omega}$ is the length metric induced by the tautological weak Finsler structure, and the Hilbert metric $H_{\Omega}$ is the length metric induced by the {\it reversible tautological} Finsler structure of the domain $\Omega$. 
This is in fact how Funk introduced his metric in 1929, see \cite{Funk,TroyanovFinslerian}. 
The reversible Finsler structure is obtained by the process of \emph{harmonic symmetrization} at the level of the convex sets in the tangent spaces that define these structures, cf. \cite{PT3}. This gives another relation between the Hilbert and the Funk metrics. The Finsler geometry of the Funk metric is studied in some details in the Chapter 3 \cite{TroyanovFinslerian} of this volume. A useful  variational description of the Funk metric is studied in \cite{Yamada2014}. There are also
interesting non-Euclidean versions of Funk geometry, see Chapter 13 in this volume \cite{PY}.  

In the present paper, we start by recalling the definition of the Funk metric, we give several basic properties of this metric, some of which are new, or at least formulated in a new way. In \S \ref{sec.ReverseFunk}, we introduce what we call the \emph{reverse Funk metric}, that is, the metric ${}^rF_\Omega$ defined by ${}^rF_\Omega(x,y)=F(y,x)$. This metric is much less studied than the Funk metric. In \S \ref{ex}, we give the formula for the Funk metric for the two main examples, namely, the case of convex polytopes and the Euclidean unit ball. In \S \ref{ref:Balls}, we study the geometry of balls in the Funk metric. Since the metric is non-symmetric, one has to distinguish between \emph{forward} and \emph{backward} balls. We show that any forward ball is the image of $\Omega$ by a Euclidean homothety. This gives a rigidity property, namely, that the local geometry of the convex set determines the convex set up to a scalar factor. In \S \ref{s:balls}, we show that the topologies  defined by the Funk and the reverse Funk metrics coincide with the Euclidean topology. In \S \ref{sec.triangle}, we give a proof of the triangle inequality for the Funk metric and at the same time we study its geodesics and its convexity properties. In \S \ref{nearest}, we study the property of the nearest point projections of a point in $\Omega$ on a convex subset of $\Omega$ equipped with the Funk metric, and the perpendicularity properties. In the case where $\Omega$ is strictly convex, the nearest point projection is unique. There is a formulation of perpendicularity to hyperplanes in $\Omega$ in terms of properties of  support hyperplanes for $\Omega$. In \S \ref{s:inf}, we study the infinitesimal (Finsler) structure associated to a Funk metric. In \S \ref{isom}, we study the isometries of a Funk metric. In the case where $\Omega$ is bounded and strictly convex, its isometry group coincides with the subgroup of affine transformations of $\mathbb{R}^n$ that leave $\Omega$ invariant.  In section \ref{s:p}  we propose a projective viewpoint and a generalization of  the Funk metric and in Section \ref{s:H}, we present some basic facts about  the Hilbert metric which are direct consequences of the fact that it is a symmetrization of the Funk metric. In the last section we give some new perspectives on the Funk metric, some of which are treated in other chapters of this volume.
Appendix A contains two classical theorems of Euclidean geometry (the theorems of Menelaus and Ceva). Menelaus' Theorem is used in Appendix B to give the classical proof of the triangle inequality for  the Funk metric.

\section{The Funk metric}   
In this section, $\Omega$ is a proper\index{proper convex domain}\index{convex domain!proper} convex domain in $\mathbb{R}^n$, that is, $\Omega$ is convex, open, 
non-empty, and $\Omega\neq \r^n$.  We denote by $\overline{\Omega}$ the closure of $\Omega$ in $\r^n $
and by $\partial \Omega = \overline{\Omega} \setminus \Omega$ the topological boundary of $\Omega$. 

In the case of an unbounded domain, it will be convenient to add points at infinity to the boundary   $\partial \Omega$. 
To do so we consider $\r^n$ as an affine space in the projective space $\mathbb{RP}^n$ and we denote
by $H_{\infty} = \mathbb{RP}^n\setminus \r^n$ the hyperplane at infinity.  We then denote by 
 $\widetilde{\Omega}$ the closure of $\Omega$ in $\mathbb{RP}^n$ and by 
 $\tilde{\partial}\Omega = \widetilde{\Omega} \setminus \Omega$. Observe that  ${\partial}\Omega = \tilde{\partial}\Omega \setminus H_{\infty}$.

For any two points $x\neq  y$ in $\mathbb{R}^n$,  we denote by $[x,y]$  the closed  affine segment joining the two points. We also denote by $R(x,y)$ the affine ray starting at $x$ and passing through $y$ and by
 $\tilde R(x,y)$ its closure in  $ \mathbb{RP}^n$. Finally we set 
 $$
  a_{\Omega}(x,y) = \tilde R(x,y)\cap \tilde \partial \Omega \in  \mathbb{RP}^n.
 $$
\begin{figure}[h]
\begin{center}
\begin{picture}(10,70)  
\thicklines 
   \closecurve(0,-10 , 70,30,  20,60)
    \put(9,5){\line(3,5){40}} 
   \put(76,34){$\Omega$ } 
     \put(50,73){\circle*{3}}
     \put(30,40){\circle*{3}}
     \put(9,5){\circle*{3}}         
     \put(45,76){$a$}
     \put(21,38.6){$y$} 
     \put(0,3.4){$x$}    
     \end{picture} 
  \end{center}  
\end{figure}

We now define the Funk metric:  
\begin{definition}[The Funk metric]\label{def:Funk}
 The \emph{Funk metric}\index{metric!Funk}\index{Funk metric} on $\Omega$, denoted by $F_{\Omega}$,  is defined  for $x$ and $y$ in $\Omega$, by $F_{\Omega}(x,x) =0$
 and by
 $$
  F_{\Omega}(x,y) =  \log \left( \frac{\vert x-a \vert}{\vert y-a \vert}\right)
 $$
 if $x\neq y$, where $a=a_{\Omega}(x,y) \in \mathbb{RP}^n$. 
 Here $\vert p-q \vert$ is the Euclidean distance between the points $q$ and $p$ in $\r^n$.
 It is understood in this formula that if $a \in H_{\infty}$, then   $F_{\Omega}(x,y)=0$. This is consistent with the
 convention $\infty/\infty = 1$.  
 \end{definition}

Let us begin with a few basic properties of the Funk metric.

\begin{proposition}\label{prop.basicproperties}
 The Funk metric in a convex domain $\Omega \neq \r^n$ satisfies the following properties:
\begin{enumerate}[{\rm (a)}]  
  \item   $F_{\Omega}(x,y) \geq 0$ and   $F_{\Omega}(x,x) = 0$   for all $x,y \in \Omega$.
  \item $F_{\Omega}(x,z) \leq F_{\Omega}(x,y)+F_{\Omega}(y,z)$   for all $x,y,z \in \Omega$.   
  \item $F_{\Omega}$ is projective,\index{metric!projective}\index{projective metric} that is, $F_{\Omega}(x,z) = F_{\Omega}(x,y)+F_{\Omega}(y,z)$
   whenever $z$ is a point on the affine segment $[x,y]$.   
  \item  The weak metric metric $F_{\Omega}$ is non symmetric, that is, $F_{\Omega}(x,y) \neq  F_{\Omega}(y,x)$
  in general.
  \item  The weak metric  $F_{\Omega}$ is separating, that is,  
  $x\neq y \Rightarrow F_{\Omega}(x,y) >  0 $,  if and only if the domain $\Omega$ is bounded.
  \item  The weak metric metric $F_{\Omega}$ is unbounded.
    \end{enumerate}
 \end{proposition}
Property (a) and (b) say that $F_{\Omega}$ is a \emph{weak metric}.

\begin{proof}
Property (a) follows from the fact that $y \in [x,a_\Omega(x,y)]$, therefore $ \frac{\vert x-a_\Omega(x,y) \vert}{\vert y-a_\Omega(x,y) \vert} \geq 1$
and we have equality if $y=x$. The triangle inequality (b) is not completely obvious.
The classical proof by Hilbert is given in Appendix B and a new proof is given in  Section \ref{sec.triangle}.

To prove (c), observe that  if $y\in [x,z]$ and $x \neq y \neq z$, then $a_{\Omega} (x,y) = 
a_{\Omega}(x,z)=a_{\Omega}(y,z)$.  Denoting this common point by $a$, we have
$$
 F_{\Omega}(x,y)+F_{\Omega}(y,z)=
 \log\frac{|x-a|}{|y-a|} + \log\frac{|y-a|}{|z-a|} =  \log\frac{|x-a|}{|z-a|}
 =  F_{\Omega}(x,z).  
$$     
Property (d) is easy to check, see Section \ref{sec.ReverseFunk} for more details.
 
Property (e) follows immediately from the definition and the fact that a convex domain $\Omega$ in $\r^n$
is unbounded if and only if it contains a ray, see \cite[Remark 3.13]{PT-Minkowski}.  

To prove (f), we recall that we always assume $\Omega\neq \mathbb{R}^n$, and therefore $\partial \Omega\neq \emptyset$. 
Let $x$ be a point in $\Omega$ and $a$ a point in $\partial\Omega$ and consider the open Euclidean segment $(x,a)$ contained in $\Omega$. For any sequence $x_n$ in this segment converging to $a$ (with respect to the Euclidean metric), we have $F_{\Omega}(x,x_n)= \log \frac{\vert x-a\vert}{\vert x_n-a\vert}\to\infty$ as $n\to\infty$. 
\end{proof}

 \medskip
 
 It is sometimes useful to see the Funk metric from  other viewpoints. If $x,y$ and $z$ are three aligned points 
 in $\r^n$ with $z\neq y$, then there is a unique $\lambda \in \r$ such that $x = z+\lambda (y-z)$. We call this number  
 the \emph{division ratio}\index{division ratio} 
 or  \emph{affine ratio}\index{affine ratio} of $x$ with respect to $z$ and $y$ and denote it suggestively by $\lambda = zx/zy$. Note 
 also that $\lambda >1$ if and only if $y\in [x,z]$. We extend the notion of  affine ratio to the case $z \in H_{\infty}$
 by setting $zx/zy = 1$ if $z  \in H_{\infty}$. We then have   
 $$
  F_{\Omega}(x,y) = \log (\lambda),
 $$
where $\lambda$ is the affine ratio of $x$ with respect to $a=a_{\Omega}(x,y)$ and $y$.

\begin{proposition} \label{prop.dFh}
Let $x$ and $y$ be two  points in the convex  domain $\Omega$ and set $a=a_{\Omega}(x,y)$.
Let  $h : \r^n \to \r$ be an arbitrary linear form such $h(y) \neq h(x)$. Then  
$$
 F_{\Omega}(x,y) =  \log \left(\frac{h(a)-h(x)}{h(a)-h(y)}\right).
$$
\end{proposition}

\begin{proof}  We first observe that a linear function $h : \r^n \to \r$ is either constant 
on a given ray $R(x,y)$, or it is injective on that ray. 
If  $a=a_{\Omega}(x,y) \in H_{\infty}$, then 
$$
 h(a) = \lim_{t\to \infty} h(x+t(y-x)) = h(x) + t (h(y) - h(x)) = \pm \infty,
$$
and the proposition is 
true since $\log(\infty/\infty) = \log(1)=0 =  F_{\Omega}(x,y)$. If 
$a \not\in H_{\infty}$, then the   division ratio $ \lambda = {ax}/{ay} > 1$ 
and we have 
$(a-x) = \lambda (a-y)$ and $F_{\Omega}(x,y) =  \log (\lambda)$. Since $h : \r^n \to \r$  is 
linear, then $(h(a)-h(x))= \lambda (h(a)-h(y))$ and the Proposition follows at once.
 \end{proof}
 
 \medskip
 
We now recall some important notions from convex geometry. For these and other classical results on convex geometry, 
the reader can consult the books by  Eggleston \cite{Eggleston},  Fenchel \cite{Fenchel1},  Valentine \cite{Valentine}
and  Rockafellar  \cite{Rockafellar}.

\begin{definition}
 Let $\Omega \subset \r^n$ be a proper convex domain and $a\in \partial \Omega$ be a (finite) boundary
 point. We assume that $\Omega$ contains the origin $0$.
 A \emph{supporting functional}\index{supporting functional} at $a$ for  $\Omega$ is a linear function $h : \r^n \to \r$ such that $h(a) =1$
 and $h(x) < 1$ for any point $x$ in $\Omega$\footnote{The reader should not confuse the notion of
 supporting functional with that of support function $\frak{h}_{\Omega}: \r^n \to \r$ defined as
 $\frak{h}_{\Omega}(x) = \sup \{ \langle x,y\rangle \mid y \in \Omega\}$. }. The hyperplane $H = \{z \mid h(z) = 1\}$ is said to
 be a \emph{support hyperplane}\index{support hyperplane} at $a$ for $\Omega$.  If $\Omega$ is unbounded, then the
 hyperplane at infinity $H_{\infty}$ is also considered to be a support hyperplane.
 A basic fact is that any boundary point of a convex domain admits one or several supporting functional(s).
 
 Let us denote by $\mathcal{S}_{\Omega}$ the set of all  supporting functionals of $\Omega$. Then 
 $$
   \frak{p}_{\Omega}(x) = \sup   \{ h(x)   \mid  {h\in \mathcal{S}_{\Omega}} \}
 $$
 is the \emph{Minkowski functional}\index{Minkowski functional} of $\Omega$. This is the unique weak Minkowski norm  such that 
 $$
  \Omega = \{ x\in \r^n \mid  \frak{p}_{\Omega}(x) < 1\}.
 $$
 \end{definition}

We have the following consequences of the previous   Proposition:

\begin{corollary}\label{cor.dFh1}
Let $\Omega \subset \r^n$ be a   convex domain  containing the origin and let 
$h$ be a    supporting functional  for $\Omega$. We then have
$$
 F_{\Omega}(x,y) \geq \log \left(\frac{1-h(x)}{1-h(y)}\right)
$$  
for any $x,y \in \Omega$.
Furthermore we have equality if and only if either $a =a_{\Omega}(x,y)  \in H_{\infty}$  and $h(x) = h(y)$,
or $a   \not\in H_{\infty}$ and \  $h(a) = 1$.
 \end{corollary}

\begin{proof}
If $h(x) = h(y)$ there is nothing to prove, we thus assume that  $h(x) \neq h(y)$.
Note that this condition means that  the line through $x$ and $y$ is not parallel
to the supporting  hyperplane $H = \{h=1\}$.

 We first assume that $a =a_{\Omega}(x,y)  \not\in H_{\infty}$, 
The hypothesis $h(x) \neq h(y)$ implies $h(a) \neq h(y)$.
Because  the function 
$a \mapsto   \log\frac{|x-a|}{|y-a|}$ is strictly monotone decreasing and   $h(a) \leq 1$, 
 we have by Proposition \ref{prop.dFh}:
$$
  F_{\Omega}(x,y)=  \log \left(\frac{h(a)-h(x)}{h(a)-h(y)}\right) \geq \log \left(\frac{1-h(x)}{1-h(y)}\right).  
 $$
The equality holds if and only if $h(a) = 1$. 

Suppose now that $a   \in H_{\infty}$; this means that  $ F_{\Omega}(x,y) =0$ and 
 the ray $R(x,y)$ is contained in $\Omega$.  In particular, we have 
$$
  h(x) + \lambda h(y-x) =  h(x+ \lambda \cdot (y-x)) < 1,  \quad  \forall \lambda \geq 0,
$$
which implies $h(y) - h(x) = h(y-x) \leq 0$. Since we assumed $h(x) \neq h(y)$, we have 
$h(y) < h(x)$ and therefore
$$
   F_{\Omega}(x,y)= 0 > \log \left(\frac{1-h(x)}{1-h(y)}\right).  
$$
 \end{proof}
 
 \medskip
 
\begin{corollary} \label{cor.dFh2}
If $\Omega \subset \r^n$ is a convex domain  containing the origin,
then
$$
 F_{\Omega}(x,y) = \max\left\{0,   \sup_{h\in \mathcal{S}_{\Omega}} \log \left(\frac{1-h(x)}{1-h(y)}\right)\right\}.
$$
\end{corollary}

\begin{proof}
We first assume that $F_{\Omega}(x,y) > 0$. Then $a = a_{\Omega}(x,y)\not\in H_{\infty}$.
Let us denote by $\Phi(x,y)$ the right hand side in the above formula. Then Proposition \ref{prop.dFh}
implies that $F_{\Omega}(x,y) \leq \Phi(x,y)$ and Corollary \ref{cor.dFh1} implies the
converse inequality. 

If $F_{\Omega}(x,y) = 0$, then $R(x,y) < 0$. 
For any 
supporting function $h$, we then   have  
$h(x+t(y-x)) =  h(x) + t(h(y)-h(x)) < 1$ for any $t>0$. This implies
that $h(y) \leq h(x)$ and it  follows that $\Phi(x,y)=0$. 
\end{proof}

\medskip

Notice that for bounded domains, the previous formula reduces to
$$
 F_{\Omega}(x,y) =  \sup_{h\in \mathcal{S}_{\Omega}} \log \left(\frac{1-h(x)}{1-h(y)}\right).
$$  
This can also be reformulated as follows (compare to \cite[Theorem 1]{Yamada2014}):  

\begin{corollary}
The Funk metric in a bounded  convex domain  $\Omega \subset \r^n$ is given by
$$
 F_{\Omega}(x,y) = \sup_H \log \left(\frac{\dist(x,H)}{\dist(y,H)}\right),
$$  
where the supremum is taken over the set of all support hyperplanes $H$ for $\Omega$
and $\dist(x,H)$ is the Euclidean distance from $x$ to $H$. 
\end{corollary}

 \medskip

The next consequence of Proposition \ref{prop.dFh} is the following relation between the division ratio
of three aligned points in a convex domain and the Funk distances between those points.

 \begin{corollary} \label{cor.dFh3}
Let $x,y$ and $z$ be three aligned points in the   convex domain   $\Omega \subset \r^n$.
Suppose $ F_{\Omega}(x,y) >0$ and $z = x + t(y-x)$ for some $t \geq 0$. Then 
\begin{eqnarray}
 \label{eq.ratio1}
  t & = & \frac {{\mathrm{e}^{F_{\Omega}(x,y)}}\cdot ( {\mathrm{e}^{F_{\Omega}(x,z)}}-1)}{{\mathrm{e}^{F_{\Omega}(x,z)}}\cdot (\mathrm{e}^{F_{\Omega}(x,y)}-1)},
  \\  \label{eq.ratio2}
 F_{\Omega}(x,z) & = & F_{\Omega}(x,y) -  \log\left(\mathrm{e}^{F_{\Omega}(x,y)} + t \cdot (1-\mathrm{e}^{F_{\Omega}(x,y)}) \right).  
\end{eqnarray}
\end{corollary}

\begin{proof}
Choose a supporting functional  $h$ for $\Omega$ at the point $a = a_{\Omega}(x,y)=a_{\Omega}(x,z)$. Then we have,
from Corollary \ref{cor.dFh1}:
$$
  \mathrm{e}^{F_{\Omega}(x,y)} =  \frac{1-h(x)}{1-h(y)}  \quad \text{ and}  \quad  \mathrm{e}^{F_{\Omega}(x,z)} =  \frac{1-h(x)}{1-h(z)}.
$$
Therefore
$$
 \frac{\mathrm{e}^{F_{\Omega}(x,z)} -1}{\mathrm{e}^{F_{\Omega}(x,z)}} 
 =  \left(\frac{1-h(z)}{1-h(x)}\right) \left(\frac{1-h(x)}{1-h(z)} - 1\right)
 = \frac{h(z)-h(x)}{1-h(x)},
$$
and likewise
$$
 \frac{\mathrm{e}^{F_{\Omega}(x,y)} -1}{\mathrm{e}^{F_{\Omega}(x,y)}} 
 = \frac{h(y)-h(x)}{1-h(x)}.
$$
Thus
 \begin{eqnarray*}
  \frac {{\mathrm{e}^{F_{\Omega}(x,y)}}\cdot ( {\mathrm{e}^{F_{\Omega}(x,z)}}-1)}{{\mathrm{e}^{F_{\Omega}(x,z)}}\cdot (\mathrm{e}^{F_{\Omega}(x,y)}-1)}
= \frac{h(z)-h(x)}{h(y)-h(x)}
= \frac{h(z-x)}{h(y-x)}
  = t.
\end{eqnarray*}
This proves Equation (\ref{eq.ratio1}). To prove the Equation (\ref{eq.ratio2}), we now resolve  
$$
  \frac{\mathrm{e}^{F_{\Omega}(x,z)} -1}{\mathrm{e}^{F_{\Omega}(x,z)}}  = t \cdot  \frac{\mathrm{e}^{F_{\Omega}(x,y)} -1}{\mathrm{e}^{F_{\Omega}(x,y)}} 
$$
for $\mathrm{e}^{F_{\Omega}(x,z)}$. This gives us
$$
\mathrm{e}^{F_{\Omega}(x,z)}= \frac {\mathrm{e}^{F_{\Omega}(x,y)}}{\mathrm{e}^{F_{\Omega}(x,y)} +t \left( 1-\mathrm{e}^{F_{\Omega}(x,y)} \right)},
$$
which is equivalent to  (\ref{eq.ratio2}).
\end{proof}

 \medskip
 
It is useful to observe that computing the Funk distance between two points in $\Omega$ is a one-dimensional operation. 
More precisely, if $S = [a_1,a_2] \subset \r^n$ is a compact segment in  $\r^n$ containing the point $x$ and $y$ in its interior
with $y \in [x,a_2]$, we shall   write
$$
 F_S(x,y) = \log\frac{|x-a_2|}{|y-a_2|}.
$$
Although $S$ is not an open set, $F_S(x,y)$  clearly corresponds to the one-dimensional Funk metric in the
relative interior of $S$.  

\begin{proposition}\label{prop.Funk1D}
 The Funk distance between two points $x$ and $y$ in $\Omega$ is given by
 $$
  F_{\Omega}(x,y) = \inf  \left\{F_S (x,y) \, \big| \ S  \text{ is a segment in $\Omega$
  containing $x$ and $y$}   \right\}.
 $$
\end{proposition}

\begin{proof}
 We identify $S$ with a segment in $\r$ with $b < x \leq y <a$ and observe that the function $a \mapsto
  \log\frac{|x-a|}{|y-a|}$ is strictly monotone decreasing. 
\end{proof}

This result can be seen as an analogy between the Funk metric and the Kobayashi metric in complex geometry, see \cite{Kobayashi}.
It has   the following immediate consequences:

\begin{corollary}
\begin{enumerate}[{\rm(i)}]
  \item If $\Omega_1 \subset \Omega_2$ are convex subsets of $\r^n$, then 
  $F_{\Omega_1} \geq F_{\Omega_2}$ with equality if and only if $\Omega_1 = \Omega_2$.
  \item Let $\Omega_1$ and $\Omega_2$ be two open convex subsets of $\mathbb{R}^n$. Then, for every $x$ and $y$ in $\Omega_{1}\cap \Omega_{2}$, we have
$F_{\Omega_{1}\cap \Omega_{2}}(x,y)=\max \left(F_{\Omega_{1}}(x,y), F_{\Omega_{2}}(x,y)\right).$
  \item Let $\Omega$ be a nonempty open
convex subset of $\mathbb{R}^n$, let $\Omega'\subset\Omega$ be the intersection of $\Omega$ with an affine subspace of $\mathbb{R}^n$, and suppose that $\Omega'\neq \emptyset$. Then, $F_{\Omega'}$ is the metric induced by $F_{\Omega}$ on  $\Omega'$ as a subspace of $(\Omega,F_{\Omega})$.
\end{enumerate}
\end{corollary}

\section{The reverse Funk metric}\label{sec.ReverseFunk} 

\begin{definition}
The \emph{reverse Funk metric}\index{reverse Funk metric}\index{Funk metric!reverse}\index{metric!reverse Funk} in a proper convex domain ${\Omega}$ is defined as
$$
 {}^rF_{\Omega}(x,y) = F_{\Omega}(y,x) = \log\left(\frac{|y-b|}{|x-b|}\right),
$$
where $b = a_{\Omega}(y,x)$.
\end{definition}

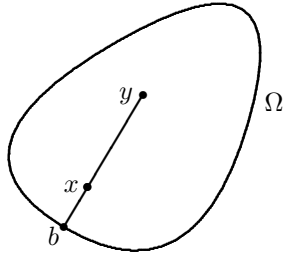
\begin{figure}[h]
\begin{center}
\begin{picture}(10,70)  
\thicklines
   \closecurve(0,-10 , 70,30,  20,60)
     \put(0,-10){\line(3,5){30}}
    \put(9,5){\circle*{3}}
   \put(76,34){$\Omega$ } 
     \put(30,40){\circle*{3}}
        \put(0,-10){\circle*{3}}
     \put(9,5){\circle*{3}}         
      \put(-6,-17){$b$}      
     \put(21,38.6){$y$} 
     \put(0,3.4){$x$}    
     \end{picture} \bigskip
   \caption{The reverse Funk Metric} 
  \end{center}  
\end{figure}

This metric satisfies the 
following properties:

\begin{proposition}\label{prop.rbasicproperties}
 The reverse Funk metric in a convex domain $\Omega \neq \r^n$ is a projective weak metric.
 It is unbounded and non-symmetric and it is separating if and only if   the domain $\Omega$ is bounded.
 \end{proposition}
 
The proof is a direct consequence of Proposition \ref{prop.basicproperties}.
\qed

\medskip

An important difference between the Funk metric and the reverse Funk metric is the following:

\begin{proposition}\label{rfunk_unbounded}  
 Let $\Omega$ be a bounded convex domain in $\r^n$ and $x$ be a point in $\Omega$. Then the function
 $y \to {}^rF_{\Omega}(x,y)$ is bounded.
\end{proposition}

\begin{proof}  Define $\lambda_x$ and $\delta$ by
$$
 \lambda_x = \inf_{b \in\partial\Omega } |x-b|, \quad and \quad 
 \delta = \sup _{a,b \in\partial\Omega } |a-b|.
$$
Observe that $\delta$ is the Euclidean diameter of  $\Omega$, thus $\delta < \infty$ since 
 $\Omega$ is bounded. We also have $\lambda_x  > 0$. The proposition follows from the
 inequality
$$
  {}^rF_{\Omega}(x,y) \leq \log \left(\frac{\delta}{\lambda_x} \right).
$$
\end{proof}

 \medskip
 
In particular the reverse Funk metric  ${}^rF_{\Omega}$ is   not bi-Lipschitz equivalent to  the Funk metric $F_{\Omega}$.  

\section{Examples} \label{ex}
 
 \begin{example}[Polytopes] \label{example.polytope}
An  (open) \emph{convex polytope}\index{convex polytope} in $\mathbb{R}^n$ is defined to be an intersection 
of finitely many half-spaces:
$$
 \Omega = \{ x \in \r^n \mid \phi_j(x) < s_j, \  1\leq j \leq k \}, 
$$
where   $\phi_j : \r^n \to \r$ is a nontrivial linear form for all $j$. 
The Funk distance between two points in such a polytope is given by
$$
 F_{\Omega}(x,y) = \max \left\{0,  \max_{1\leq j \leq k}  \log \left(\frac{s_j-\phi_j(x)}{s_j-\phi_j(y)}\right)\right\}.
$$  
The proof is similar to that of Corollary  \ref{cor.dFh2}.   As a special case, let us mention
that the Funk metric in $\r_+^n$ is given by
$$
   F_{\r^n_+}(x,y) = \max_{1 \leq i \leq n}\max \left\{ 0,\log \frac{x_i}{y_i}  \right\}.
$$
Observe that the map $x = (x_i) \mapsto u = (u_i)$, where $u_i = \log (x_i)$, is an isometry from 
the space  $(\r_+^n,F_{\r^n_+}) $   to $\r^n$
with the weak Minkowski distance 
$$
  \delta (u,v) = \max_{1 \leq i \leq n}\max \left\{ 0, u_i-v_i  \right\}.
$$
 \end{example}

 \begin{example}[The Euclidean unit ball] 
 The following is a formula for the Funk metric in the Euclidean unit ball  $B \subset \mathbb{R}^n$:
\begin{equation}\label{eq.FunkinBall}
   F_{B}(x,y)=  \log \left(
   \frac{\sqrt{ |y-x|^2 - |x\wedge y|^2} +  \vert x\vert^2  - 
   \langle x,y\rangle}{\sqrt{ |y-x|^2 - |x\wedge y|^2} -  \vert y\vert^2 + \langle x,y\rangle}
   \right),
\end{equation}
where $|x\wedge y|=\sqrt{ |x|^2|y|^2 -  \langle x,y\rangle^2}$ is the area of the parallelogram with sides 
$\overrightarrow{0x}, \overrightarrow{0y}$.
\end{example}

\begin{center}
\begin{tikzpicture}
   \draw (0,0) circle (1.8cm);
    \filldraw[lightgray] (-0.8,1) -- (0.6,1 ) -- (0,0)  -- cycle;
    \put(0,0){\circle*{2}} ;
  \draw (-0.8,1) -- (1.43,1 );   
 \filldraw   (1.5,1) circle  (0.03cm)  ;
 \filldraw   (-0.8,1) circle  (0.03cm)  ; 
 \filldraw   (0.6,1) circle  (0.03cm)  ;  
 \draw [->] (0,0) -- (1.5,0); 
 \draw [->] (0,0) -- (0,1); 
  \draw [->] (0,0) --  (1.5,1 )  ;
  \draw (1.6,1.2) node {\small{$a$}};
  \draw (-0.8,1.2) node  {\small{$x$}};
  \draw (0.6,1.2) node  {\small{$y$}}; 
  \draw (0,-0.2 ) node  {\small{$0$}}; 
  \draw (-0.2,0.6) node  {\small{$a_2$}};   
  \draw (1.2,-0.2) node  {\small{$a_1$}};     
  \end{tikzpicture}
\end{center}
\begin{proof}
 If $x=y$, there is nothing to prove, so we assume that $x\neq y$.
Let us set $a= a_B(x,y) = R(x,y) \cap \partial B$. 
Using   Proposition \ref{prop.dFh} with the linear form $h(z)= \langle y-x,z\rangle$ 
we get
$$
   F_{B}(x,y) =  \log \left(
   \frac{\langle y-x,a\rangle - \langle y-x,x\rangle}{\langle y-x,a\rangle - \langle y-x,y\rangle} \right)
 =  \log \left(
   \frac{\langle y-x,a\rangle  +  \vert x\vert^2  - \langle x,y\rangle}{ \langle y-x,a\rangle  -  \vert y\vert^2  + \langle x,y\rangle} \right).  
$$
So we just need to compute $\langle y-x,a\rangle$. This is an exercise in elementary Euclidean geometry.
Let us set  $u = \frac{y-x}{|y-x|}$ and 
$$
  a_1 = \langle u,a\rangle u, \qquad  a_2 = a-a_1.
$$ 
Then $a=a_1+a_2$ and 
$a_1$ is a multiple of $y-x$  while  $a_2$ is the orthogonal projection of the origin $O$ of $\mathbb{R}^n$ on the line through $x$ and $y$.
In particular the height of the triangle $Oxy$ is equal to $|a_2|$, therefore
$$
 \mbox{Area}(Oxy) = \frac 12 |x\wedge y| =  \frac 12 |a_2| \cdot |y-x|.
$$ 
Observe now  that $\langle u,a\rangle   >0$ and $|a|^2 =  |a_1|^2 +  |a_2|^2 =1$, we thus have
\begin{eqnarray*} 
  \langle y-x,a\rangle^2 & = & |y-x|^2 \cdot  \langle u,a\rangle^2 
  \\ & = & |y-x|^2 \cdot  |a_1|^2   \\ & = &   |y-x|^2 \cdot (1- |a_2|^2)
  \\ & = &   |y-x|^2 - |x\wedge y|^2
\end{eqnarray*}
The desired formula  follows immediately. \end{proof}

\section{The geometry of  balls in the Funk metric}\label{ref:Balls}

Since we are dealing with non-symmetric distances, we need to distinguish between \emph{forward} and \emph{backward}  balls.
For a point  $x$ in $\Omega$ and $\rho>0$, we set
\begin{equation} \label{eq:right}
B^+(x,\rho)= \{y\in B \mid F_\Omega(x,y)< \rho\}
\end{equation} and we call it the   \emph{forward open ball}\index{forward open ball}\index{ball!forward}
(or \emph{right open ball}\index{right open ball}\index{ball!right}) centered at $x$ of radius $\rho$.
In a symmetric way, we set
\begin{equation} \label{eq:left}
B^-(x,\rho)= \{y\in B  \mid F_\Omega(y,x)< \rho\}
\end{equation}
and we call it the \emph{backward open ball}\index{backward open ball}\index{ball!backward}
(also called the \emph{left open ball}\index{left open ball}\index{ball!left}) centered at $x$ of radius $\rho$.  

Note that the open backward balls of the Funk metric are the open forward balls of the reverse Funk metric,
and vice versa.
 
We define closed  {forward} and closed  {backward} balls  by replacing the inequalities in (\ref{eq:right}) and (\ref{eq:left}) by non strict inequalities,
 and in the same way we define   {forward} and  {backward}  spheres\index{forward sphere}\index{sphere!forward}\index{backward sphere}\index{sphere!backward} by replacing the inequalities by equalities.
In Funk geometry, the backward and forward balls have in general quite different shapes and different properties.

\begin{proposition}\label{prop:balls-homothetic}  
Let $\Omega$ be a proper convex open subset of $\mathbb{R}^n$ equipped with its Funk metric $F_{\Omega}$,  let $x$ be a point in $\Omega$ and let $\rho$ be a nonnegative real number. We have: \\
$\bullet $  ÊThe forward open ball $B^+(x,\rho) $ is the image of $\Omega$ by the Euclidean homothety   of center $x$ and dilation factor $(1-e^{-\rho})$. \\
$\bullet $ ÊThe  backward   open ball $B^-(x,\rho)$ is the intersection of $\Omega$ with the   image of $\Omega$ by the Euclidean homothety of center $x$ and dilation factor $(e^{\rho}-1)$, followed by the Euclidean central symmetry centered at $x$. 
 \end{proposition}

 \begin{center}
    \vspace{1cm}  \begin{picture}(100,100)    \centering 
         \includegraphics[width=.40\linewidth]{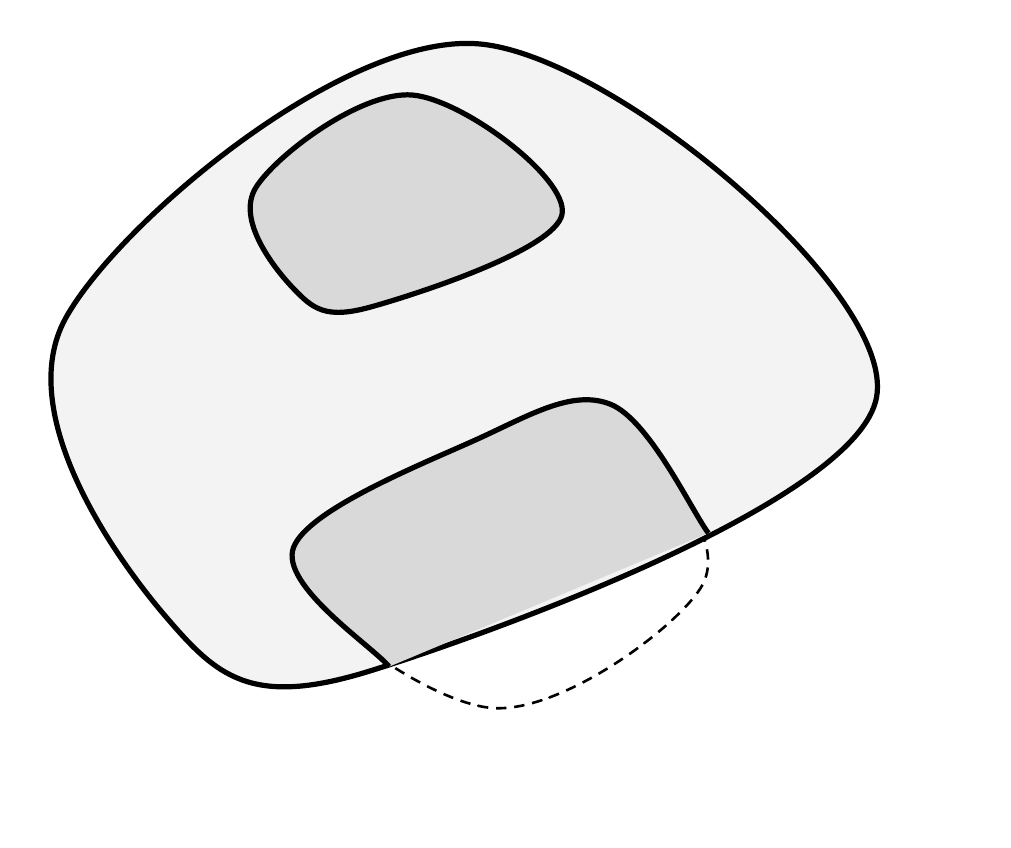}  
     \put(-22,45){$\Omega$ } 
   \put(-84,94){\circle*{3}}      
      \put(-65,47){\circle*{3}}      
   \end{picture}  
       \begin{minipage}[h]{0.85\textwidth}  
{\small {A forward and a  backward ball in Funk geometry. The forward ball is always relatively 
compact in $\Omega$, while the closure of the backward ball may meet the boundary $\partial \Omega$ if its radius is large enough.}}
 \end{minipage}
 \end{center}

  \begin{proof}
Let $y\neq x$ be a point in $\Omega$. If $F_{\Omega}(x,y)  = 0$, then the 
ray $R(x,y)$ is contained in $\Omega$, and for any $z$ on that ray, we have $F_{\Omega}(x,z)=0$.
Therefore the  ray is also contained in $B(x,\rho)$.  
If  $F_{\Omega}(x,y)  = 0$, then $a = a_{\Omega}(x,y) \neq H_{\infty}$ and we have the  following equivalent conditions for any point $y$ on the segment $[x,a]$:  
 \begin{eqnarray*}
 y\in B^+(x,\rho) & \Leftrightarrow &  \log\frac{\vert x-a\vert}{\vert y-a\vert} < \rho
 \\ & \Leftrightarrow & 
 \vert x-a\vert <  e^{ \rho} \vert y-a\vert
= e^{ \rho}( \vert x-a\vert- \vert y-x\vert )
 \\ & \Leftrightarrow &  \vert y-x\vert < (1-e^{-\rho}) \vert x-a\vert. 
\end{eqnarray*} 
This proves the  first statement. The proof of the second statement is similar,
let us set $b = a_{\Omega}(y,x)$, then   for any point $y \in [x,a]$ we have
$x \in [b,y]$, therefore
 \begin{eqnarray*}
 y\in B^-(x,\rho) & \Leftrightarrow &  \log\frac{\vert y-b\vert}{\vert x-b\vert} < \rho
 \\ & \Leftrightarrow & 
 e^{ \rho} \vert x-b\vert  >  \vert y-b\vert =  \vert y-x\vert +  \vert x-b\vert
  \\ & \Leftrightarrow & 
   (e^{ \rho}-1) \vert x-b\vert  >   \vert y-x\vert .
   \end{eqnarray*} 
\end{proof}

 Thus, for instance, if $\Omega$ is the interior a  Euclidean ball in $\mathbb{R}^n$, then any  forward ball for the 
Funk metric $B^+(x_0,\delta)$ is also a Euclidean ball. However, its Euclidean center  is not the center for the Funk metric (unless $x_0$ is the center of $\Omega$).
Considering  Example 4.2,  if  $\Omega$ is the Euclidean unit ball and 
$B^+(x_0,\rho ) \subset \Omega$ is  the Funk ball off radius $\rho  $ and center $x_0$ in $\Omega$, then $y \in B^+(x_0,\rho) $ if and only if
$F_{\Omega}(x_0,y) \leq \rho $. Using   Formula (\ref{eq.FunkinBall}), we compute that this  is equivalent to
$$
  \|y\|^2 - 2e^{-\rho}\langle y,x_0\rangle + e^{-2\rho  }\|x_0\|^2 \leq  (1-e^{-\rho})^2.
$$
This set describes  a Euclidean ball with center $z_0 = e^{-\rho  }x_0$ and Euclidean radius $r= (1-e^{-\rho})$
 
 \smallskip
 We deduce the following ``local-implies-global" property of Funk metrics. The meaning of the statement is clear, and it follows directly from Proposition \ref{prop:balls-homothetic}.
  
  \begin{corollary} 
 We can reconstruct the boundary $\partial\Omega$ of $\Omega$ from the local geometry at any point of $\Omega$.
\end{corollary}
    
\begin{corollary}
For any points $x$ and $x'$ in a convex domain $\Omega$ equipped with its Funk metric and for any two positive real numbers $\delta$ and $\delta'$, the forward balls $B^+(x,\delta)$ and $B^+(x',\delta')$ are either homothetic
or a translation of each other. 
 \end{corollary}
  
\begin{proof}
 This follows from Proposition \ref{prop:balls-homothetic} and the fact that the set of Euclidean transformations
 which are either homotheties or translations form a group (sometimes called the \emph{the group of
 dilations}, see e.g. \cite{Coxeter}).
\end{proof}

 \begin{remarks}
The previous Corollary also holds for backward balls  $B^-(x,\delta)$ of small enough radii. 
\end{remarks}

\begin{remarks}
$\bullet$  In the case where the convex set $\Omega$ is unbounded, its forward and  backward open balls of the Funk metric are always noncompact. 
\\ $\bullet$  If $\Omega$ is bounded. then for any  $x\in \Omega$ and for $\rho$ large enough we have
$B^-(x,\rho) = \Omega$. This follows from Proposition \ref{rfunk_unbounded}. In particular, the closed
backwards balls are not compact for large radii.
\\ $\bullet$ The forward open balls are geodesically convex if and only if $\Omega$ is strictly convex.
\end{remarks}

 \begin{remark} The property for a weak metric on a subset $\Omega$ of $\mathbb{R}^n$ to have all the right spheres homothetic is also shared by the Minkowski weak metrics on $\mathbb{R}^n$. Indeed, it is easy to see that in a Minkowski weak metric, any two right open balls are homothetic. (Any two right spheres of the same radius are translates of each other, and it is easy to see from the definition that any two spheres centered at the same point are homothetic, the center of the homothety being the center of the balls.) Thus, Minkowski weak metrics share with the Funk weak metrics the property stated in Proposition \ref{prop:balls-homothetic}. Busemann proved that in the setting of \emph{Desarguesian spaces}, these are  the only examples of spaces satisfying this property (see the definition of a Desarguesian space and the statement of this result in Chapter 1,  Section 6 \cite{PT-Minkowski} in this volume). We state this as the following:
 \end{remark}
  
\begin{theorem}[Busemann  \cite{Busemann-homothetic}]
 A Desarguesian space in which all the right spheres of positive radius around any point are homothetic is either a Funk space or a Minkowski space.
\end{theorem}

\section{On the topology of the Funk metric} \label{s:balls}

\begin{proposition}\label{prop.topology}
The topology induced by the Funk or reverse Funk metric in a bounded convex domain 
 $\Omega$  in  $\mathbb{R}^n$ coincides 
with the Euclidean topology in that domain.
 \end{proposition}

\begin{proof}
The proof consists in comparing the balls in the Euclidean and the Funk (or reverse Funk)
geometries. Let us fix a point $x$ in $\Omega$. Then there exists $0 < \lambda_x \leq  \Lambda_x <\infty$
such that for any $\xi \in \partial\Omega$ we have
$$
 \lambda_x \leq  |\xi-x| \leq \Lambda_x.
$$
If we denote by $B^+(x,\rho)$ the forward ball with center $x$ and radius $\rho$ in the Funk metric, then 
Proposition  \ref{prop:balls-homothetic} implies that 
$$
  y \in \partial B^+(x,\rho) \  \Rightarrow \    (1-\mathrm{e}^{-\rho})\lambda_x \leq  |y-x| \leq (1-\mathrm{e}^{-\rho}) \Lambda_x. 
$$
In other words, if \  ${}^\mathrm{E}B(x, \delta)$ denotes the Euclidean ball  with center $x$ and radius $\delta$, then 
\begin{equation}
\label{eq.compareballs}  
 {}^\mathrm{E}B(x,(1-\mathrm{e}^{-\rho}) \lambda_x )  \subset B^+(x,\rho) \subset  {}^\mathrm{E}B(x,(1-\mathrm{e}^{-\rho})\Lambda_x ). 
\end{equation} 
This implies that the families of balls $B^+(x,\rho)$ and   ${}^\mathrm{E}B(x, \delta)$ are sub-bases for the same topology.

For the backward balls $B^-(x,\rho)$, the second part of Proposition  \ref{prop:balls-homothetic} implies the following
$$
  y \in \partial B^-(x,\rho) \  \Rightarrow \   (e^{\rho}-1)\lambda_x \leq  |y-x| \leq  (e^{\rho}-1)\Lambda_x,
$$
provided $(e^{\rho}-1) \leq 1$. This implies that for $\rho \leq \log (2)$ we have
$$
 {}^\mathrm{E}B(x,(e^{\rho}-1)\lambda_x )  \subset B^-(x,\rho) \subset  {}^\mathrm{E}B(x,(e^{\rho}-1)\Lambda_x ), 
$$
and therefore the family of backward balls $B^-(x,\rho)$ also generates the Euclidean topology.
\end{proof}

For general convex domains, bounded or not, we have the following weaker result on the topology: 

\begin{proposition}\label{eq:conv-F} For any  convex domain $\Omega$ in $\mathbb{R}^n$,
 $ F_{\Omega}$ is a continuous functions on $\Omega\times \Omega$.
 \end{proposition}

\begin{proof} We first consider the case where $\Omega$ is bounded. 
Suppose first that $x$  and $y$ are distinct points in $\Omega$ and let $x_n,y_n$ be sequences in $\Omega$ converging to $x$ and $y$ respectively. 
Taking subsequences if necessary, we may assume that $x_n \neq y_n$ for all 
$n$. Then $a_n =a_{\Omega} (x_n,y_n)$ is well defined and this sequence converges to $a =a_{\Omega}(x,y)$. Since  
$a \neq y$  we have
\begin{eqnarray*}
 \lim_{n \to \infty} F_{\Omega}(x_n,y_n)  &=&  \lim_{n \to \infty}  \log\left(\frac{|x_n-a_n|}{|y_n -a_n|} \right)
 \\ &=&
 \log  \left(\frac{|x-a|}{|y-a|} \right)  \\ &=& F_{\Omega}(x,y).
\end{eqnarray*}
 Assume now that $x= y$ and let  $x_n,y_n \in \Omega$ be sequences converging to $x$ such that 
 $x_n \neq y_n$ for all $n$. We have
 $$
   F_{\Omega}(x_n,y_n) =  \log\left(\frac{|x_n-a_n|}{|y_n -a_n|} \right)
  $$
  $$= \log\left(\frac{|(y_n-a_n) + (x_n-y_n)|}{|y_n -a_n|} \right)
 \leq  \log\left(1+\frac{|x_n-y_n|}{|y_n -a_n|} \right).
$$
Since $y_n \in \Omega$ converges to a point $x$ in $\Omega$, we have
$ \delta = \sup_{b \in \partial \Omega} |y_n -b|^{-1} <\infty$. We then have
$$
   F_{\Omega}(x_n,y_n)   \leq  \log\left(1+ \delta |x_n-y_n| \right) \to 0,  
$$
since $|x_n-y_n|\to 0$.

If  $\Omega$ is unbounded, we set  $\Omega_R = \Omega \cap   {}^\mathrm{E}B(x, R)$ where  ${}^\mathrm{E}B(x,R)$ 
is the Euclidean ball of radius $R$ centered at the origin. It is easy to check that 
$F_{\Omega_R}$  converges uniformly to $F_{\Omega}$ on every compact subset of 
$\Omega \times \Omega$ as $R \to \infty$. The continuity of  $F_{\Omega}$ follows therefore
from the proof  for bounded convex domains.
\end{proof}

For the next result we need some more definitions:
 
\begin{definition}
 Let $\delta$ be a weak metric defined on a set $X$. A sequence $\{x_k\}$  in $X$ is \emph{forward bounded}\index{metric!forward bounded}\index{forward bounded!metric}
 if 
 $$
  \sup\delta(x_k,x_m) < \infty
 $$
 where the supremum is taken over all pairs $k,m$ satisfying $m\geq k$.
 Note that this definition corresponds to the usual notion in the case of a  usual (symmetric) metric space.  
 We then say that the weak metric space $(X,\delta)$ is \emph{forward proper},\index{metric!forward proper}\index{forward proper metric}
 or \emph{forward boundedly compact}\index{metric!forward boundedly compact}\index{forward boundedly compact metric}
 if every  forward bounded sequence has a converging subsequence.  \\
 The sequence $\{x_k\}$ is  \emph{forward Cauchy}\index{forward Cauchy sequence}
 if
 $$
  \lim_{k\to \infty} \  \sup_{m \geq k}\delta(x_k,x_m) = 0.
 $$ 
 The weak metric space $(X,\delta)$ is \emph{forward complete} if every forward Cauchy
 sequence converges.  We define \emph{backward properness}\index{metric!backward proper}\index{backward proper metric}
 and \emph{backward completeness}\index{metric!backward complete}\index{backward complete metric} in a similar way.
 \end{definition}

\begin{proposition} \label{prop.Fcomplete}  
The Funk metric in a convex domain $\Omega \subset \r^n$ is forward proper (and in particular forward complete) if and only if  $\Omega$ is bounded. The Funk metric is never backward complete.
\end{proposition}
 
\begin{proof} For a convex domain, the ball inclusions (\ref{eq.compareballs})
immediately imply that forward complete balls are relatively compact; this implies 
forward properness. If $\Omega$ is unbounded, then it contains a ray and such a
ray contains a divergent sequence $\{x_k\}$ such that
$F_{\Omega}(x_k,x_m) = 0$ for any $m \geq k$ therefore $F_{\Omega}$ is not 
complete. 

To prove that the Funk metric is never backward complete, we consider an affine
segment $[a,b] \in \r^n$ with $a \neq b$ and such that  
$ [a,b] \cap \partial \Omega = \{a,b\}$. Set  $x_k = b + \frac{1}{k} (a-b)$. If $m \geq k$,
then 
$$
  F_{\Omega}(x_k,x_m) = \log \frac{|x_m-a|}{|x_k-a|},
$$
which converges to $0$ as $k\to \infty$. Since the sequence $\{x_k\}$ has no limit in 
$\Omega$, we conclude that $F_{\Omega}$ is not backward complete. 
\end{proof}
 
\begin{remark}  \label{rem.Fcomplete}  
The previous proposition also says that in a bounded convex domain, the Funk metric is
forward complete and the reverse Funk metric is not.  
\end{remark}

\section{The Triangle inequality and geodesics} \label{sec.triangle}

In this section we prove the triangle inequality for the Funk metric and give a necessary and
sufficient condition for the equality case. We also describe all the geodesics of this metric.

\subsection{On the triangle inequality}
 
\begin{theorem}\label{th.triangle_ineq}
If $x$, $y$ and $z$ are three points in a proper convex domain $\Omega$,
then the triangle inequality 
\begin{equation} \label{ineq.triangle}
 F_{\Omega}(x,y) + F_{\Omega}(y,z) \geq  F_{\Omega}(x,z) 
\end{equation}
holds. Furthermore we have equality 
$
 F_{\Omega}(x,y) + F_{\Omega}(y,z) = F_{\Omega}(x,z) 
$
if and only if the three points
\begin{equation}\label{abcendpoints}
   a_{\Omega}(x,y), \, a_{\Omega}(y,z), \, a_{\Omega}(x,z) \in \tilde \partial \Omega 
\end{equation}
are aligned in $ \mathbb{RP}^n$.  
 \end{theorem}

Before proving this theorem, let us first recall a few additional definitions from convex geometry: 
Let $\Omega \subset \r^n$ be a convex domain. Then it is known that its closure $\overline{\Omega}$ is also convex.  
A convex subset $D \subset \overline{\Omega}$ is a \emph{face}\index{face (of a convex set)} of $\overline{\Omega}$ if for any $x,y\in D$ and any $0<\lambda<1$
we have
$$
 (1-\lambda) x +  \lambda  y \in D   \  \Rightarrow \  [x,y] \subset D.
$$
The empty set and  $\overline{\Omega}$ are also considered to be faces. A face  $D \subset \overline{\Omega}$ is
called \emph{proper} if  $D \neq \overline{\Omega}$ and $D\not=\emptyset$.  A face $D$ is said to be \emph{exposed}\index{exposed face} if there
is a supporting  hyperplane $H$ for $\Omega$ such that $D = H\cap  \overline{\Omega}$. Recall that a support
hyperplane is a hyperplane $H$ that meets $\partial \Omega$ and $H\cap \Omega = \emptyset$.
It is easy to prove that every proper face is contained in an exposed face. In fact every maximal proper face is exposed.

\medskip

A point $x\in \partial \overline{\Omega}$ is an \emph{exposed}\index{exposed point} point of $\Omega$ if $ \{x\}$ is an exposed
face, that is, if there exists a hyperplane $H \subset \r^n$ such that $H \cap \overline{\Omega} = \{x\}$.
If $\Omega$ is bounded, then  $\overline{\Omega}$ is the closure of the convex hull of its exposed points (Straszewicz's Theorem).

\medskip

A point $x \in \overline{\Omega}$ is an \emph{extreme point} if 
$\overline{\Omega} \setminus \{x\}$ is still a convex set. Such a point belongs to the boundary $\partial \Omega$
and if $\Omega$ is bounded, then  $\overline{\Omega}$ is the convex hull of its extreme points (Krein-Milman's Theorem).\index{Theorem!Krein-Milman}
Every exposed point is an extreme point, but the converse does not hold in general. The following result immediately
follows from the definitions:

\begin{lemma}
The following are equivalent conditions for a convex domain $\Omega \subset \r^n$:
\begin{enumerate}[{\rm (i.)}]
  \item Every boundary point is a extreme point.
  \item Every boundary point is an exposed point.
  \item The boundary
$\partial \Omega$ does not contain any non-trivial segment. 
\end{enumerate}
\end{lemma}
If one of these conditions holds, then   $\Omega$ is said to be \emph{strictly convex}\index{strictly convex set}.
The following result will play an important role in the proof of Theorem \ref{th.triangle_ineq}:

\begin{lemma}\label{lem.pointfacealigned}
 Let $\Omega$ be bounded convex domain and $x,y,z$ three points in $\Omega$.
  Then the following are equivalent:
\begin{enumerate}[{\rm (a)}]
  \item There exists a proper face $D\subset \tilde\partial \Omega$ such that 
  $$
   a_{\Omega}(x,y), \, a_{\Omega}(y,z), \, a_{\Omega}(x,z) \in D.
  $$
  \item The three points $ a_{\Omega}(x,y), \, a_{\Omega}(y,z)$ and $a_{\Omega}(x,z)$
  are aligned in $\mathbb{RP}^n$.
  \end{enumerate} 
  \end{lemma}

\begin{proof}
 Let us set $a = a_{\Omega}(x,y)$,
  $b = a_{\Omega}(y,z)$ and  $c = a_{\Omega}(x,z)$. If $x,y$ and $z$ are aligned,
  then $a=b=c$. Otherwise, $a$, $b$ and $c$ belong to the $2$-plane $\Pi$ containing $x,y,z$.
  Therefore if $a,b,c$ belong to a proper face $D$, then   those three points are
  contained   in the interval $\Pi\cap D$. This proves the implication (a) $\Rightarrow$ (b)  
  The converse implication  (b) $\Rightarrow$ (a) is obvious.    
\end{proof}

\bigskip

\textbf{Proof of Theorem \ref{th.triangle_ineq}.}
We first consider $F_{\Omega}(x,z) = 0$. In this case the inequality (\ref{ineq.triangle}) is trivial and
we have $c \in H_{\infty}$. We then have equality in  (\ref{ineq.triangle}) if and only if 
$F_{\Omega}(x,y) = F_{\Omega}(y,z) = 0$,  which is equivalent to $a  \in H_{\infty}$ and $b \in H_{\infty}$.
It then follows from Lemma \ref{lem.pointfacealigned} that $a,b,c$ lie on some line (at infinity).

\smallskip 

We now consider the case $F_{\Omega}(x,z) > 0$, that is, $c \not\in H_{\infty}$. Choose a supporting
functional $h$ at the point $c$ (that is, $h(c) = 1$). We then have from  Corollary  \ref{cor.dFh1}:
\begin{eqnarray*} 
  F_{\Omega}(x,z)  &=&  \log \left(\frac{1-h(x)}{1-h(z)}\right)
   \\ &=&   \log \left(\frac{1-h(x)}{1-h(y)}\right)
  +  \log \left(\frac{1-h(y)}{1-h(z)}\right) 
  \\ & \leq &
    F_{\Omega}(x,y) + F_{\Omega}(y,z).  
\end{eqnarray*}
This proves the triangle inequality. Using again  Corollary  \ref{cor.dFh1}, we see that we have 
equality if and only if
$$
 F_{\Omega}(x,y)  =  \log \left(\frac{1-h(x)}{1-h(y)}\right) \quad \mbox{and} \quad
 F_{\Omega}(y,z) =  \log \left(\frac{1-h(y)}{1-h(z)}\right), 
$$
and this holds if and only if one of the following cases holds:

\smallskip 

\emph{Case 1.} We have   $a \not\in H_{\infty}$ and  $b \not\in H_{\infty}$. 
In that case, $h(a) = h(b) = 1 = h(c)$. The three points $a,b,c$ belong to the face
$D =  \partial\Omega \cap \{h=1\}$ and we conclude by  Lemma \ref{lem.pointfacealigned} 
that $a,b,c$ lie on some line .

\smallskip 

\emph{Case 2.} We have   $a  \in H_{\infty}$ and  $b \not\in H_{\infty}$. 
In that case $h(b) = 1 = h(c)$  and $h(x) = h(y)$. This implies that the line through
$x$ and $y$ is parallel to the hyperplane $\{h=1\}$ and therefore the point 
$$
 a \in \tilde R(x,y) \cap H_{\infty} \subset  \tilde{\partial}\Omega
$$
belongs to the support hyperplane $\{h=1\}$. Since $h(b) = 1$, the 
three points $a,b,c$ belong to the face
$D =  \tilde{\partial}\Omega \cap \{h=1\}$ and we conclude by  Lemma \ref{lem.pointfacealigned}. 
 
 \smallskip 

\emph{Case 3.} We have $a  \not\in H_{\infty}$ and  $b  \in H_{\infty}$. 
The argument is the same as in Case 2.

 \smallskip 
 
To complete the proof, we need to discuss the case   $a \in H_{\infty}$ and  $b  \in H_{\infty}$. 
In this case, we would have  $h(x) = h(y)$ and $h(y) = h(z)$ and this is not possible. Indeed, we have
$c = x+ \lambda (z-x)$ for some $\lambda$ and the equality $h(z) = h(x)$
would lead to the  contradiction
$$
 1 = h(c) = h(x+ \lambda (z-x)) =  h(x) + \lambda (h(z)-h(x))  = h(x)  < 1.
$$ 

We thus proved in all cases that the equality holds in  (\ref{ineq.triangle})
if and only if the points $a$, $b$ and $c$ are aligned in $\mathbb{RP}^n$.
\qed

\medskip

\begin{corollary}\label{cor.strictexposed}
Let $x$ and $z$ be two points in a  proper convex domain  $\Omega \subset \r^n$. 
Suppose that $a_{\Omega}(x,z)\in \partial \Omega$ is an exposed point.
Then for any point $y \not\in [x,z]$ we have 
$F_{\Omega}(x,z) < F_{\Omega}(x,y) + F_{\Omega}(y,z)$.
\end{corollary}
 
 \medskip

\subsection{Geodesics and convexity in Funk geometry }

 We now describe  geodesics in Funk geometry. Let us start with a few definitions.
 
\begin{definitions}
A \emph{path} in a  weak  metric space $(X,d)$ is a continuous map $\gamma : I\to X$, 
where $I$ is an  interval of the real line. The \emph{length} of path $\gamma : [a,b]\to X$
is defined as
$$
 \mbox{Length}(\gamma) = \sup \sum_{i=0}^{N-1} d(\gamma(t_i),\gamma(t_{i+1})),
$$
where the supremum is taken over all subdivisions $a = t_0 < t_1 < \cdots < t_N = b$.
 Note that in the case  where the weak metric $d$ is non-symmetric, the order of the arguments  is important. 
 The path   $\gamma : [a,b]\to X$ is a  \emph{geodesic}\index{weak metric!geodesic} if
$d(\gamma(a),\gamma(b)) =   \mbox{Length}(\gamma)$. 
The weak metric space $(X,d)$ is said to be a \emph{weak geodesic metric space} if there exists a 
geodesic path connecting any pair of points. It is said to be \emph{uniquely geodesic} if this
geodesic path is unique up to reparametrization. 
A subset $A \subset X$ is  said to be \emph{geodesically convex}\index{geodesically convex}\index{metric!geodesically convex}  if given any two points in $A$, any geodesic path joining them is contained in $A$.
\end{definitions}

\begin{lemma}\label{lem.critgeodesic}
The path  $\gamma : [a,b]\to X$ is geodesic if and only if  for any  $t_1,t_2,t_3$ in $[a,b]$ satisfying $t_1 \leq t_2 \leq t_3$ we have $d(\gamma(t_1),\gamma(t_3))=d(\gamma(t_1),\gamma(t_2)) +d(\gamma(t_2),\gamma(t_3))$. 
\end{lemma}

 The proof is an easy consequence of the definitions. \qed
 
 \medskip
 
Let us now consider a proper convex domain $\Omega \subset \r^n$ and a proper face $D \subset \tilde{\partial} \Omega$.
For any point $p\in \Omega$, we denote by  
\begin{equation}\label{eq.defcone}
 C_p(D) = \{v \in \r^n \mid v = 0 \text{ or } \overline{R}(p,p+v)\cap D \neq \emptyset\}.
\end{equation}
Here $\overline{R}(p,p+v)$ is the extended ray through $p$ and $p+v$ in $\mathbb{RP}^n$.  (Recall that the projective space $\mathbb{RP}^n$ is considered here as a completion of the Euclidean space $\mathbb{R}^n$ obtained by adding a hyperplane at infinity; the completion of the ray is then its topological completion.)
Observe that $C_p(D)$ is a cone in $\r^n$ at the origin, its translate $p+C_p(D)$ is 
the cone over $D$ with vertex at $p$.
We then have the following

\begin{theorem}\label{th.FunkGeodesics}
Let $\gamma : [0,1]\to  \Omega$ be a path in a proper convex domain of  $ \r^n$. 
Then $\gamma$ is a geodesic    for the Funk metric in $\Omega$ if and only if there exists a
face $D \subset \tilde{\partial} \Omega$ such that 
for any $t_1 < t_2$ in $[0,1]$ we have 
$$\gamma(t_2)-\gamma(t_1) \in  C_{\gamma(t_1)}(D).$$
\end{theorem}

In particular if $a_{\Omega}(x,y) \in \partial\Omega$ is an exposed point, then 
there exists a unique (up to reparametrization) geodesic joining $x$ to $y$,
and this geodesic is a parametrization of the affine segment $[x,y]$.    

\begin{proof}
 This is a direct consequence of Theorem \ref{th.triangle_ineq} together with 
Lemma \ref{lem.pointfacealigned} and Lemma
\ref{lem.critgeodesic}. 
\end{proof}

For smooth curves we have the following 
\begin{corollary}
Let  $\gamma : [0,1]\to  \Omega$ be a $C^1$ path in a proper convex domain of  $ \r^n$. 
Then $\gamma$ is a Funk geodesic if and only if there exists a face $D \subset \tilde{\partial} \Omega$ such that 
$\dot \gamma (t) \in  C_{\gamma(t)}(D)$ for any $t \in [a,b]$.
\end{corollary}

 \begin{center}
    \vspace{1cm} 
     \begin{picture}(130,120)    \centering 
         \includegraphics[width=.40\linewidth]{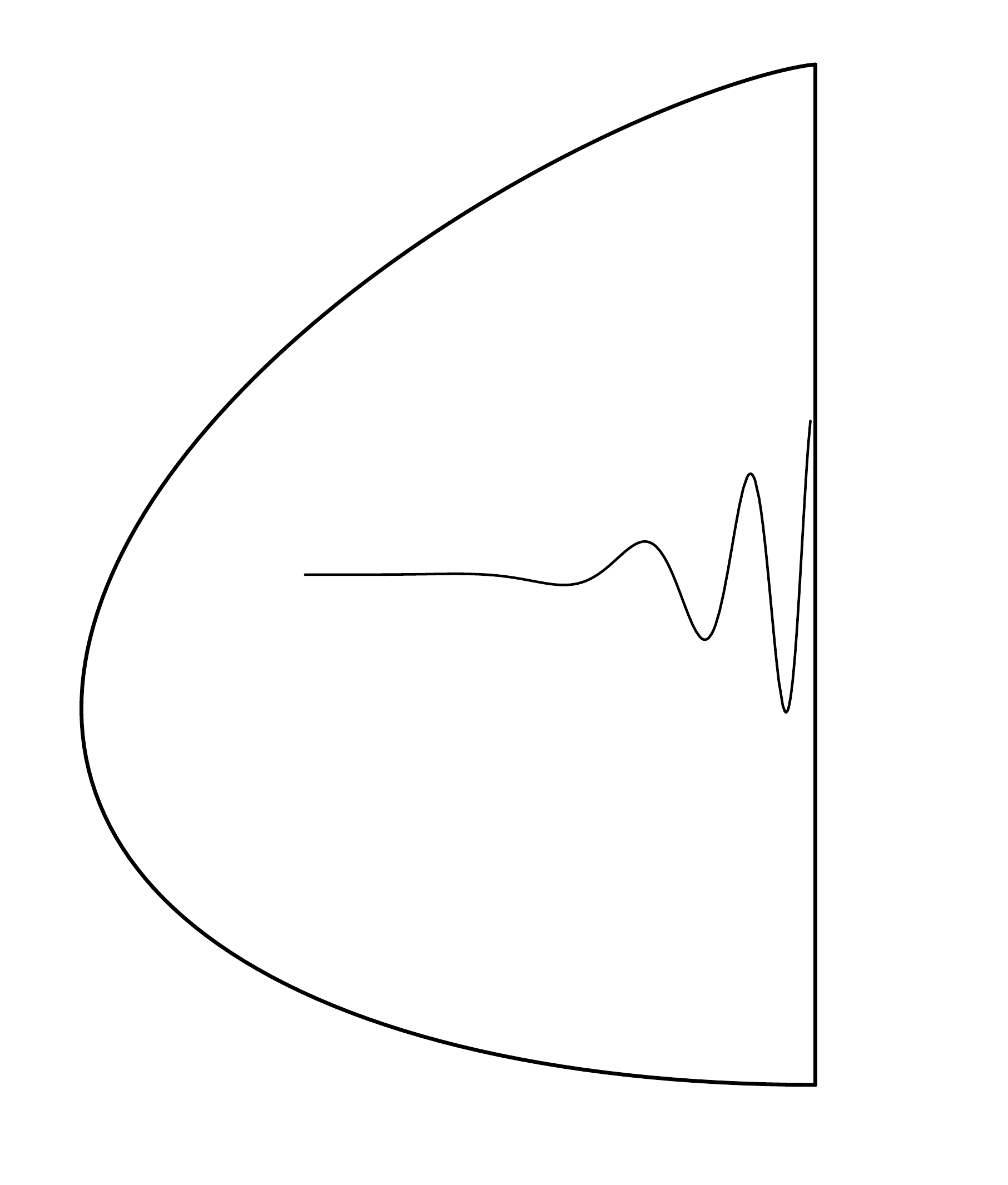}  
     \put(-22,45){$D$ }       \put(-86,75){$\gamma$ } 
   \put(-120,99){$\Omega$ }     
   \end{picture}  \\
  \begin{minipage}[h]{0.75\textwidth}  
    {\small {A typical smooth geodesic in Funk geometry: all tangents to the curve meet the same face $D \subset\partial \Omega$.}}
  \end{minipage}
 \end{center}  

\medskip 

For a subset $\Omega$ of $\mathbb{R}^n$, equipped with a (weak) metric $F$, we have two notions of convexity:  \emph{affine convexity},\index{affinely convex}\index{convex!affinely} saying that for every pair of points in $\Omega$, the affine (or Euclidean) geodesic joining them is contained in $\Omega$, and \emph{geodesic convexity},\index{geodesically convex}\index{convex!geodesically} saying that for every pair of points in $\Omega$, the $F$-geodesic joining them is contained in $\Omega$

From the preceding results, we have the following consequence on geodesic  convexity of subsets for the Funk metric.

\begin{corollary}\label{cor.setgeodcomplete}
Let  $\Omega$ be a bounded convex domain of  $ \r^n$. Then the following are equivalent:
\begin{enumerate}
  \item $\Omega$ is strictly convex.
  \item $\Omega$ is uniquely geodesic for the Funk metric.
  \item A subset $A\subset \Omega$ is geodesically convex for the Funk metric
  if and only if $A$ is affinely convex.
  \item The forward  open balls in $\Omega$ are geodesically convex with respect to the Funk metric $F_{\Omega}$.
\end{enumerate}
\end{corollary}
 
 \begin{proof}
(1) $\Rightarrow$ (2) immediately follows from Theorem  \ref{th.FunkGeodesics}. (2) $\Rightarrow$ (3)
is obvious and (3) $\Rightarrow$ (4)  follows from Proposition \ref{prop:balls-homothetic}.
 
\begin{tikzpicture}[line cap=round,line join=round,>=triangle 45,x=0.7cm,y=0.7cm]
\clip(-2,-4.2) rectangle (10,7);
  \draw plot  [smooth ]  coordinates  {  
(0,0) (-0.3,2.5)    (4.2,4.2) (8.4,2.5)  (8,0) }; 
\draw  (0,0)--(8,0);
  \draw plot  [smooth ]  coordinates  {  
(4,0.4) (3.8875,1.3375)    (5.575,1.975) (7.15,1.3375)  (7,0.4) }; 
\draw  (5.5,0)--(5.5, 2);
\draw [line width=0.2pt,dash pattern=on 3pt off 1.1pt]  (4,0.4)--(7,0.4);
\draw [line width=0.2pt,dash pattern=on 3pt off 1.1pt]  (5.5,1.8)--(0.8,0);
\draw [line width=0.2pt,dash pattern=on 3pt off 1.1pt]  (3.5,1.02)--(7.5,0);
\draw [line width=0.2pt,dash pattern=on 3pt off 1.1pt]  (3.5,1.02)--(5.5,1.02);

 \begin{scriptsize}
 \draw  (6.6,3.8) node {$\Omega$};
\fill   (5.5,0.52 ) circle (0.8pt);
\draw  (5.7,0.6)  node {$y$};
\fill   (3.5,1.02) circle (0.8pt);
\draw  (3.35,1.15)  node {$q$};
\fill   (5.5,1.02) circle (0.8pt);
\draw  (5.7,1.0)  node {$p$};
\fill   (5.5,1.8) circle (0.8pt);
\draw  (5.7,1.8)  node {$x$};

 \draw  (8,-0.2) node {$b$};
 \draw  (-0.1,-0.2) node {$a$};
\end{scriptsize}
\end{tikzpicture}
 
We now prove  (4) $\Rightarrow$ (1) by contraposition. Let us assume that  $\Omega$ is not strictly convex, so that there exists a non-trivial segment 
$[a,b] \subset \partial \Omega$. On can then find  a supporting functional $h$ for $\Omega$ such that 
$h(a) = h(b) = 1$.  Let us chose a segment $[x,y] \subset \Omega$ such that $a_{\Omega}(x,y) \in [a,b]$
and another segment  $[p,q] \subset  \Omega$ such that  $p \in [x,y]$ and $h(q) = h(p)$. 
We also assume $x \neq p \neq y$ and $F_{\Omega}(p,q) >  \delta :=  F_{\Omega}(p,x)+F_{\Omega}(p,y) $.
Observe that we then have $h(x) < h(p) = h(q) < h(y)$. 

If $a_{\Omega}(x,q) \in [a,b]$ and $a_{\Omega}(q,y) \in [a,b]$, then the proof is finished since
in this case
$$
  F_{\Omega}(x,y) = \log \frac{h(x)}{h(y) } =  \log \frac{h(x)}{h(q)} \cdot \frac{h(q)}{h(y)} = 
  F_{\Omega}(x,q) +   F_{\Omega}(q,y). 
$$
Since $x,y \in B^+(p,\delta)$  while $q \not\in B^+(p,\delta)$, we conclude that the forward ball
$B^+(p,\delta)$ is not geodesically convex.

If $a_{\Omega}(x,q) \not\in [a,b]$ or  $a_{\Omega}(q,y) \not\in [a,b]$, we let $c = a_{\Omega}(x,y)$
and we   consider the Euclidean homothety 
$f_{\lambda} : \r^n \to \r^n$ centered at $c$ with dilation factor $\lambda < 1$. Let $p' = f_{\lambda}(p)$, $q' = f_{\lambda}(q), x'=f_{\lambda}(x), y' = f_{\lambda}(y)$ and $q' = f_{\lambda}(q)$. It is now clear that if $\lambda > 0$ is small enough, then $a_{\Omega}(x',q') \in [a,b]$ and $a_{\Omega}(q',y') \in [a,b]$. It is also clear that one can find a number $\delta'$
such that  $x',y' \in B^+(p',\delta')$  while $q' \not\in B^+(p',\delta')$. The 
 previous argument shows  then that 
$  F_{\Omega}(x',y')  =   F_{\Omega}(x',q') +   F_{\Omega}(q',y')$ and therefore $B^+(p',\delta')$ is not
geodesically convex. 
\end{proof}

 \begin{remark}
 Note the formal analogy between Corollary \ref{cor.setgeodcomplete}  and the corresponding result concerning the geodesic segments of a Minkowski metric on $\mathbb{R}^n$: if the unit ball of a Minkowski  metric is strictly convex, then the only geodesic segments of this metric are the affine segments.
\end{remark}

\section{Nearest points in Funk Geometry}\label{nearest}

Let $\Omega\subset \mathbb{R}^n$ be a convex set equipped with its Funk metric $F$.

\begin{definition}
Let $x$ be a point in $\Omega$ and let $A$ be a subset of $\Omega$. A point $y$ in $A$ is said to be a
\emph{nearest point},\index{nearest point}  or a  \emph{foot}\index{foot} (in Buseman's terminology), for $x$ on $A$ if 
$$
  F_{\Omega}(x,y) =  F_{\Omega}(x,A) := \inf_{z \in A} F_{\Omega}(x,z).
$$ 
\end{definition}
It is clear from the continuity of the function $y \mapsto  F_{\Omega}(x,y)$ that for any closed non-empty subset 
$A \subset \Omega$ and any $x\in \Omega$, there exists a nearest point $y\in A$. This point need not be unique
in general. 

\begin{proposition}
For a proper convex domain $\Omega \subset \r^n$, the following properties are equivalent:
\begin{enumerate}[a.)]
\item $\Omega$ is strictly convex.
\item For any closed convex subset $A \subset \Omega$ and for any $x \in \Omega$, there is a unique
nearest point $y\in A$.
\end{enumerate}

\begin{proof}
let $x$ be a point in $\Omega$ and assume  $r= F_{\Omega}(x,A)>0$. 
Suppose that $y$ and $z$ are two nearest points of $A$ for $x$. 
For  any point $w$ on the segment $[y,z]$ we have $F_{\Omega}(x,w) \leq r$
because the  closed ball $\bar{B}^+(x,r)$ is convex.
Since $A$ is also assumed to be convex, we have $w \in A$ and
therefore $F_{\Omega}(x,w) \geq r$.  We conclude that  $F_{\Omega}(x,w) = r$, that is, 
$w \in \partial \bar{B}^+(x,r)$.  From Proposition \ref{prop:balls-homothetic}, we know that   if  $\Omega$ is strictly convex, then $\bar{B}^+(x,r)$ 
is also strictly convex and we conclude that $y=z$. It follows that we have a 
unique  nearest point on $A$ for $x$. This proves \  (a) $\Rightarrow$ (b).

To prove  \  (b) $\Rightarrow$ (a),   we assume 
by contraposition that $\Omega$ is strictly convex. Again from Proposition \ref{prop:balls-homothetic}, we know that the forward ball $\bar{B}^+(x,r)$
is not strictly convex. In particular  $\partial {B}^+(x,r)$ contains a non trivial 
segment $A= [y,z]$. Any point in the convex set $A$  is a nearest point to $x$ and
this completes the proof.
\end{proof}
\end{proposition}

\begin{proposition}
 Let $A$ be an affinely convex closed subset of a proper convex domain  $\Omega \subset \r^n$ and let $x\in \Omega \setminus A$.
 A point $y\in A$ is a nearest point in $A$ for $x$ if and only if either $F_{\Omega}(x,y) = 0$ or there exists a hyperplane
 $\Pi \subset \r^n$ which contains $y$, which separates $A$ and $x$ and which is parallel to a  
 support hyperplane $H$  for $\Omega$ at $a = a_{\Omega}(x,y)$.
\end{proposition}

\begin{proof}
We assume $F_{\Omega}(x,y) > 0$ (otherwise, there is nothing to prove).  First, suppose there exists a hyperplane
$\Pi \subset \r^n$ containing $y$ and separating $A$ from $x$ and which is parallel to a  
 support hyperplane $H$  for $\Omega$ at $a$. Let $h$ be the corresponding supporting functional. Then
we have, from our hypothesis,
$$
  h(x) < h(y) =  \inf_{z \in A}  h(z).
$$
From Proposition \ref{prop.dFh} and Corollary \ref{cor.dFh1} we then have
$$
   F_{\Omega}(x,y) =  \log \left(\frac{1-h(x)}{1-h(y)}\right) = \inf_{z \in A} \log \left(\frac{1-h(x)}{1-h(z)}\right)
   \leq F_{\Omega}(x,A),
$$
therefore $y$ is a nearest point on $A$ for $x$.

To prove the converse, we assume that $y$ is a nearest point on $A$ for $x$. Set $r =  F_{\Omega}(x,y) = F_{\Omega}(x,A)$,
then, by definition, the forward open ball $B^+(x,r)$ and the set $A$ are disjoint. Since both sets are affinely convex, there exists a 
hyperplane $\Pi$ that separates them. Note that  $\Pi$ is then a support hyperplane at $y$ for the ball $B^+(x,r)$. 
We conclude from Proposition \ref{prop:balls-homothetic} that $\Pi$  is parallel to a  
support hyperplane $H$  for $\Omega$ at $a$.
\end{proof}

\medskip

There are several possible notions of perpendicularity in metric spaces. The following
definition is due to Busemann (see \cite[page 103]{Busemann1955}). 

\begin{definition}[Perpendicularity]\label{def:perpendicularity} 
Let $A$ be a subset of $\Omega$ and $p$ a point in $A$. 
A  geodesic  $\gamma : I \to\Omega$ is said to be 
\emph{perpendicular}\index{perpendicular} to $A$ at $p$ if the following two properties hold:
\begin{enumerate}
\item $p = \gamma (t_0)$ for some $t_0 \in I$, 
\item  for every $t  \in I$,  $p$ is a nearest point for $\gamma(t)$ on $A$.
\end{enumerate}
\end{definition}

From the previous results we then have the following

\begin{corollary}
 Let $x$ be a point in a convex domain  $\Omega$  and $a\in \partial \Omega$  be a boundary point.
 If $\Pi \subset \r^n$ is a hyperplane containing $x$, then the ray $[x,a)$  is perpendicular to $\Pi \cap \Omega$
 if and only if $\Pi$ is parallel to a support hyperplane $H_a$ of $\Omega$ at $a$.
 
 If $b\in \partial \Omega$  is another boundary point, then the line $(a,b)$ is   perpendicular to $\Pi \cap \Omega$
 if and only if $\Pi \cap (a,b) \neq \emptyset$ and 
 $\Pi$ is parallel to both a support hyperplane  $H_a$  at $a$ and a  support hyperplane  $H_b$  at $b$.
\end{corollary}

\section{The infinitesimal Funk distance} \label{s:inf}
In this section, we consider a convex domain $\Omega\subset \r^n$  and a point $p$ in $\Omega$.
We define a weak distance $\Phi_p = \Phi_{\Omega,p}$ on $\r^n$ as the limit
$$
  \Phi_p(x,y) = \lim_{t\searrow 0} \ \frac{F_{\Omega,p}(p+tx,p+ty)}{t}.
$$
\begin{theorem}
The weak metric $\Phi_p$  at a point $p$ in $\r^n$ is a Minkowski weak metric in $\r^n$.
Its unit ball is the translated domain $\Omega_p = \Omega -  p$.
\end{theorem}

\begin{proof}
Choose a supporting functional $h$ for $\Omega$ and set   
$$
 \varphi_h (t) = \log\left(\frac{1-h(p+tx)}{1-h(p+ty)} \right)
 = \log\left(1+\frac{th(y-x)}{1-h(p)- th(y)} \right).
$$
The first two derivatives of this functions are
$$
 \varphi_h'(t) =  \frac {\left(1-{h(p)} \right)  h(y-x) }{ \left(1-h(p)- th(x) \right)  \left( 1-h(p)- th(y)\right)},
 $$
and
$$   
  \varphi_h''(t) =   \frac {  \left((1-h(p))(h(x)+h(y)) - 2t \cdot  h(x) h(y)\right) \left(1-{h(p)} \right)  h(y-x)}
  {\left(1-h(p)- th(x) \right)^2  \left( 1-h(p)- th(y)\right)^2}.
$$
We have in particular
$$
  \varphi'_h (0) =  \frac{h(y-x)}{1- h(p)},
$$
and the second derivative is  uniformly bounded in some neighborhood of $p$.
More precisely, given a relatively compact neighborhood $U\subset \overline{U} \subset  \Omega$
of $p$, one can find a constant $C$ which depends on $U$ but not on $h$ such that
$$
 | \varphi_h''(t)| \leq  C,
$$ 
for any $x,y \in  U$  and $|t| \leq 1$ and any support function $h$. We have, from Taylor's formula,
$$
 \varphi_h (t) =  t \cdot  \frac{h(y-x)}{1- h(p)} + t^2\rho(x,y,h),
$$
where $|\rho(x,y,h)| \leq C$.  
Using Corollary \ref{cor.dFh2} we have
$$
 F_{\Omega,p}(p+tx,p+ty) = \sup_h  \varphi_h (t) =  \sup_h    \frac{t  h(y-x)}{1- h(p)}  + O(t^2),
$$
where the supremum is taken over  the set $\mathcal{S}_{\Omega} $ of all support functions for $\Omega$.
Therefore 
$$
 \Phi_p(x,y) =     \sup_{h\in \mathcal{S}_{\Omega}} \frac{h(y-x)}{1- h(p)}.
$$
We then see that $ \Phi_p(x,y) $ is weak Minkowski distance (see Chapter 1 in this handbook).
\begin{eqnarray*}
 \Phi_p(0,y) \leq 1 &  \Leftrightarrow & \sup_{h \in \mathcal{S}_{\Omega}}  \frac{h(y) }{1- h(p)}  \leq  1
  \\ &  \Leftrightarrow &
h(y) \leq 1 - h(p) \text{ for all support functions $h$ of $\Omega$}
  \\ &  \Leftrightarrow &
h(p+y) \leq 1  \text{ for all   $h$}
  \\ &  \Leftrightarrow &  
 y \in \overline{\Omega}-p, 
\end{eqnarray*}
this means that the unit ball of $\Phi_p$ is the translate of $\Omega$ by $-p$.
\end{proof}

\medskip

\begin{remark}
The Funk metric of a convex domain $\Omega$
 is in fact Finslerian, and the previous theorem means that the Finslerian unit ball
 at any point $p$  coincides with the domain $\Omega$ itself with the point $p$
 as its center. This is why the Funk metric was termed \emph{tautological}\index{tautological!Finsler structure} in \cite{PT2}.

 The Finslerian approach to Funk geometry is developed in
 \cite{TroyanovFinslerian}. 
 \end{remark}

\section{Isometries} \label{isom}

It is clear from its definition that the Funk metric is invariant under affine transformation. Conversely, we have the following:

\begin{proposition}\label{prop.Isometriesl} 
Let $\Omega_1$ and $\Omega_2$ be two bounded convex domains in $\r^n$.  Assume that there exists a Funk isometry
$f : U_1 \to U_2$, where $U_i$ is an open convex subset of $\Omega_i$,  ($i=1,2$). If $\Omega_2$ is strictly convex, then 
$f$ is the restriction of  a global affine map of $\r^n$ that maps $\Omega_1$ to  $\Omega_2$.
\end{proposition}

\begin{proof}
Let $x$ and $y$ be two distinct points in $U_1$. Then, for any $z \in [x,y]$, we have
\begin{equation*}
\begin{split}
  F_{\Omega_2}(f(x),f(y)) -  F_{\Omega_2}(&f(x),f(z)) -  F_{\Omega_2}(f(z),f(y)) 
\\  &= \    
F_{\Omega_1}(x,y) - F_{\Omega_1}(x,z) - F_{\Omega_1}(z,y) = 0.
\end{split}
\end{equation*}
Since $\Omega_2$ is assumed to be strictly convex, this implies that $f(z)$ belongs to the line
through $f(x)$ and $f(y)$.  We can thus define the real numbers $t$ and $s$ by
$$
 z = x+t(y-x),  \quad  \text{and}  \quad  f(z) = f(x) + s(f(y)-f(x)).
$$
It now follows from Corollary  \ref{cor.dFh3} and
the fact that $f$ is an isometry that $s=t$.
We thus have proved that for any $x,y \in U_1$ and any $t \in [0,1]$ we have
\begin{equation}\label{ }
 f(x+t(y-x)) =  f(x) + t(f(y)-f(x)).
\end{equation}
This relation easily implies that $f$ is the restriction of a global affine mapping.
\end{proof}

We immediately deduce the following:

\begin{corollary}\label{cor:iso} 
The group of Funk isometries of a strictly convex  bounded  domains  $\Omega \subset \r^n$
coincides with the subgroup of the affine group of $\r^n$ leaving $\Omega$ invariant.        
\end{corollary}

\begin{remark}
The conclusion of this corollary may fail for unbounded domains. If for instance $\Omega$
is the upper half plane $\{ x_2 >0\}$ in $\r^2$, then $F_{\Omega}(x,y) = \max\{0, \log (x_2/y_2)\}$
and any map $f : \Omega \to \Omega$ of the type $f(x_1,x_2) = (ax_1+b, \psi (x_2)$, where $a\neq 0$
and $\psi : \r \to \r$ is arbitrary, is an isometry.
\end{remark}
 
\section{A projective viewpoint on Funk geometry}\label{s:p}

In this section we consider the following generalization of Funk geometry:
We say that a subset $U \subset \mathbb{RP}^n$ is \emph{convex} if
it does not contain any full projective line and 
if the 
intersection of any projective line $L \subset  \mathbb{RP}^n$ with $U$ is a connected set. 
If  $U$ and $\Omega$   are connected domains in $\mathbb{RP}^n$ with 
$\Omega \subset U$, and if $x,y$ are two distinct points in $\Omega$. We denote
by $a_{\Omega}(x,y) \in \partial \Omega$ and $\omega_{\Omega}(x,y) \in \partial U$
the boundary points on the line $L$ through $x$ and $y$ appearing in the order
$a,y,x,\omega$.

\begin{definition} The \emph{relative Funk metric}\index{relative Funk metric}\index{Funk metric!relative}\index{metric!relative Funk} 
for $\Omega \subset U$ is defined
by $ F_{ \Omega,U}(x,x) =0$
 and by  
 $$
  F_{\Omega,U}(x,y) =  \log  \left(\frac{\vert y-\omega \vert}{\vert x-\omega \vert}  \cdot \frac{\vert x-a \vert}{\vert y-a \vert} \right),
 $$
 if $x\neq y$. 
 \end{definition} 
 
The relative Funk metric is a projective weak metric, it is  invariant under projective transformations in the sense that if $f : \mathbb{RP}^n \to \mathbb{RP}^n$ is a 
projective transformation, then
$$
    F_{\Omega,U}(x,y) =   F_{f(\Omega),f(U)}(f(x),f(y)). 
$$
Observe also that if \ $U\subset \r^n$ is a proper convex domain, then
$$
    F_{\Omega,U}(x,y) =   F_{\Omega }(x,y) + {}^r F_{U}(x,y). 
$$ 
 
 \begin{lemma}
In the case $U = \r^n$, we have 
$$
    F_{\Omega,U}(x,y) =   F_{\Omega }(x,y).
$$ 
\end{lemma}
 
\begin{proof}
 If $U = \r^n$, then its boundary is the hyperplane at infinity $H_{\infty}$ and thus
$\frac{\vert y-\omega \vert}{\vert x-\omega \vert} = 1$ for any $x,y \in \Omega$. 
\end{proof}

Recall that there is no preferred hyperplane in projective space. Therefore, \emph{the classical Funk geometry is a special case of the relative Funk geometry where the 
englobing domain $U \subset \mathbb{RP}^n$ is the complement of a hyperplane.} Such a set $U$ is sometimes called an \emph{affine patch}.\index{affine patch} 

\section{Hilbert geometry}\label{s:H}
 
\begin{definition} The \emph{Hilbert metric}\index{Hilbert metric} in 
a proper convex domain  $\Omega \subset \r^n$ is defined as
$$
  H_{\Omega}(x,y) =  \frac{1}{2}  \left(F_{\Omega }(x,y) + {}^r F_{\Omega}(y,x)\right),
$$
where $\Omega$ is considered as a subset of an affine patch  $U\subset \mathbb{RP}^n$.
 \end{definition} 
This metric is a projective weak metric. Note that for $x\neq y$ we have 
$$
    H_{\Omega}(x,y) =  
    \frac{1}{2} \log  \left(\frac{\vert y-b \vert}{\vert x-b \vert}  \cdot \frac{\vert x-a \vert}{\vert y-a \vert} \right),
$$
where $a=_{\Omega }(x,y)$ and $b=a_{\Omega }(y,x)$.
The expression inside the logarithm is the cross ratio of the points $b,x,y,a$,
therefore the Hilbert metric is invariant by projective transformations.

\begin{figure}[h]
\begin{center}
\begin{picture}(10,70)  
\thicklines
   \closecurve(0,-10 , 70,30,  20,60)
   \put(0,-10){\line(3,5){50}}
   \put(0,-10){\circle*{3}}
   \put(76,34){$\Omega$ } 
     \put(50,73){\circle*{3}}
     \put(30,40){\circle*{3}}
     \put(9,5){\circle*{3}}         
     \put(-6,-17){$b$}      
     \put(45,76){$a$}
     \put(21,38.6){$y$} 
     \put(0,3.4){$x$}    
     \end{picture} \bigskip
   \caption{The Hilbert metric} 
  \end{center}  
\end{figure}
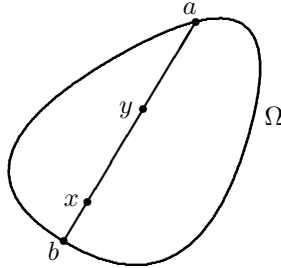

Note also that the Hilbert metric coincides with one half the relative Funk
metric of the domain $\Omega$ with respect to itself:
$$
  H_{\Omega}(x,y) =  \frac{1}{2}  F_{\Omega,\Omega}(x,y), 
$$

In the case where $\Omega$ is the Euclidean unit ball, the Hilbert distance coincides with the
Klein model\index{Klein model} (also called the \emph{Beltrami-Cayley-Klein model}\index{Beltrami-Cayley-Klein model}) of hyperbolic space. We refer to Sections 2.3--2.6 in  \cite{Thurston}
for a nice  introduction to Klein's model.
In the case of a convex polytope  defined by the linear inequalities 
$\phi_j(x) < s_j$, \  $1\leq j \leq k$,  we have
$$
 H_{\Omega}(x,y) = \max_{1\leq i, j \leq k}  \frac{1}{2}\cdot \log 
 \left(\frac{s_i-\phi_j(y)}{s_i-\phi_j(x)} \cdot \frac{s_j-\phi_j(x)}{s_j-\phi_j(y)}\right).
$$  

\medskip

Applying our investigation on Funk geometry immediately gives a number of
results on Hilbert geometry. In particular, applying Proposition \ref{prop.basicproperties}
we get:

\begin{proposition}\label{prop.Hbasicproperties}
 The Hilbert metric in a convex domain $\Omega \neq \r^n$ satisfies the following properties:
\begin{enumerate}[{\rm (a)}]  
  \item   $H_{\Omega}(x,y) \geq 0$ and   $H_{\Omega}(x,x) = 0$   for all $x,y \in \Omega$.
  \item $H_{\Omega}(x,z) \leq H_{\Omega}(x,y)+H_{\Omega}(y,z)$   for all $x,y,z \in \Omega$.   
  \item $H_{\Omega}$ is projective, that is, $H_{\Omega}(x,z) = H_{\Omega}(x,y)+H_{\Omega}(y,z)$
   whenever $z$ is a point on the affine segment $[x,y]$.   
  \item  The weak metric metric $H_{\Omega}$ is   symmetric, that is,  $H_{\Omega}(x,y) = H_{\Omega}(y,x)$ for any $x$ and $y$.
  \item  The weak metric  $H_{\Omega}$ is separating, that is,  
  $x\neq y \Rightarrow H_{\Omega}(x,y) >  0 $,  if and only if the domain $\Omega$ does not contain any affine line.
  \item  The weak metric metric $H_{\Omega}$ is unbounded.
    \end{enumerate}
 \end{proposition}
 
The proof of this proposition easily follows from the definitions and from  Proposition \ref{prop.basicproperties}.

\begin{proposition}
 If the  convex domain $\Omega \neq \r^n$  does not contain any affine line, then $H_{\Omega}$ is a metric
 in the classical sense. Furthermore, it is complete.
\end{proposition}

 A convex domain  which does not contain any affine line is called a \emph{sharp} convex domain.\index{sharp convex domain}\index{convex domain!sharp}

\medskip

From Theorem \ref{th.triangle_ineq}, we deduce the following necessary and sufficient condition for the equality case in the triangle inequality:

\begin{theorem} 
Let $x$, $y$ and $z$ be three points in a proper convex domain $\Omega$. We have
$
 F_{\Omega}(x,y) + F_{\Omega}(y,z) = F_{\Omega}(x,z) 
$
if and only if both triple of boundary points  $a_{\Omega}(x,y), a_{\Omega}(y,z), a_{\Omega}(x,z) $
and $a_{\Omega}(y,x), a_{\Omega}(z,y), a_{\Omega}(z,x) $
are aligned in $ \mathbb{RP}^n$.  
 \end{theorem}

From Theorem \ref{th.FunkGeodesics}, we obtain:

\begin{theorem}
Let $\gamma : [0,1]\to  \Omega$ be a path in a sharp convex domain of  $\r^n$. 
Then $\gamma$ is a geodesic    for the Hilbert metric in $\Omega$ if and only if there exist two
faces $D^-,D^+ \subset \tilde{\partial} \Omega$ such that 
for any $t_1 < t_2$ in $[0,1]$ we have $\gamma(t_2)-\gamma(t_1) \in  C_{\gamma(t_1) }(D^+)$ and 
$\gamma(t_1)-\gamma(t_2) \in  C_{\gamma(t_2) }(D^-)$.
\end{theorem}

Recall that $C_p(D)$ is the cone at $p$ on the face $D$, see Equation (\ref{eq.defcone}).
If $\gamma$ is smooth, then it is geodesic if and only $\dot \gamma (t) \in  C_{\gamma(t)}(D^+)$ 
and $-\dot \gamma (t) \in  C_{\gamma(t)}(D^-)$ for any $t \in [0,1]$. 
We then have the following characterization of smooth geodesics in Hilbert geometry
which we formulate only for bounded domain for convenience:

\begin{corollary}
Let $\gamma : [0,1]\to  \Omega$ be a path  of class $C^1$ in a bounded convex domain of  $\r^n$. 
Then $\gamma$ is a geodesic    for the Hilbert metric in $\Omega$ if and only if there exist two
faces $D^-,D^+ \subset \tilde{\partial} \Omega$ such that for any $t$, the tangent line to the curve $\gamma$
at $t$ meets the boundary $\partial \Omega$ on $D^+ \cup D^-$.
\end{corollary}

\begin{corollary}[compare  \cite{Harpe1993}]
 Assume that $\Omega$ is strictly convex, or more generally that all but possibly one of its proper faces are reduced to points. 
 Then the Hilbert geometry in $\Omega$ is uniquely geodesic.
\end{corollary}

\medskip

\section{Related questions}\label{other}
 
In this section we briefly discuss some recent developments related to the idea of the Funk distance.
 
 Yamada  recently introduced what he called the \emph{Weil-Petersson-Funk metric}\index{Weil-Petersson-Funk metric}\index{metric!Weil-Petersson-Funk} on Teichm\"uller space (see \cite{Yamada2014}, and see \cite{OMY} in this volume). The definition is  analogous to one of the definitions of the Funk metric, using in an essential way the non-completeness of the Weil-Petersson metric on Teichm\"uller space. One can wonder whether there are Funk-like metrics associated to other interesting known symmetric metrics. This geometry bears some analogies with Thurston's metric on Teichm\"uller space and the Funk metric. One can ask for a study of the Thurston metric which parallels the study of the Funk metric (that is, study its balls, its convexity properties, orthogonality and projections, etc.). 
 
It should also be of interest to study the geometric properties of the \emph{reverse} Funk metric,\index{reverse Funk metric}\index{metric!reverse Funk} that is, the metric ${}^rF_\Omega$ on an open convex set $\Omega$ defined by ${}^rF_\Omega(x,y)=F_\Omega(y,x)$.  
Let us recall that 
 the reverse Funk metric is not equivalent to the Funk metric in any reasonable sense, see Remark \ref{rem.Fcomplete}.      This is also related to the fact that the forward and backward open balls at some point can be very different, as we already noticed.
We note in this respect that the \emph{reverse} metrics of the Thurston weak metric and of the Weil-Petersson-Funk weak metric that we mentioned are also very poorly understood.
 
\smallskip
 
 Finally, there is another symmetrization of the Funk metric, besides the Hilbert metric, namely, its \emph{max-symmetrization},\index{metric!max-symmetrization}\index{max-symmetrization!metric} defined as 
 $$S(x,y)=\max \{(F(x,y),F(y,x)\},$$
  and it should be interesting to study its propertes. Note that the max-symmetrization of the Thurston metric is an important  metric on Teichm\"uller space, known as the \emph{length spectrum metric}\index{length spectrum metric}\index{metric!length spectrum}. 

 \appendix
\section{Menelaus' Theorem}\index{Theorem!Menelaus}

An elementary proof of the triangle inequality for the Funk metric is given by  Zaustinsky  in \cite{Zaustinsky}, and it is recalled in the next appendix. This proof is based on 
the  classical Menelaus' Theorem.\footnote{This theorem, in the Euclidean and in the spherical case, is quoted by Ptolemy (2nd c. A.D) and it is due to Menelaus (second c. A.D.) Its proof is contained in 
Menelaus' \emph{Spherics}. No Greek version of Menelaus'  \emph{Spherics} survived, but there are Arabic versions; cf. the forthcoming English edition \cite{RP} from the Arabic original of al-Haraw\=\i \ (10th century).} 
For the convenience of the reader, we give  a statement and a proof of this result in the present appendix.

 \medskip
 
To state Menelaus' Theorem, we recall   the notion of \emph{division ratio} of three aligned points.
Consider three points $A,B,P$ in $\r^n$ with  $A\neq B$. Then $P$ belongs to the line through $A$ and $B$
if and only if   $P = t B + (1-t)A$ for some uniquely defined $t\in \r$. 
The number $t$ is called the \emph{division ratio}\index{division ratio} or the \emph{affine ratio}\index{affine ratio} of 
$P$ relative to $B$ and $A$. We denote it by
$t = \frac{AP}{AB}$. The division ratio is invariant under any affine transformation.
Note that if both $A\neq B$ and $A \neq P$, then 
$$
   \frac{AP}{AB} = t \quad  \Longleftrightarrow \quad  \frac{PB}{PA} = \frac{t-1}{t}.
$$
For instance $P$ is the midpoint of $[B,A]$ if and only if $\frac{AP}{AB}  = \frac{1}{2}$ or, equivalently, $\frac{PB}{PA} =-1$. 
Note that the sign is an important component of the division ratio, and in fact we have
$$
    \frac{AP}{AB}  = \pm \frac{|P-A|}{|B-A|} 
$$
with a minus sign if and only if $A$ lies between $B$ and $P$.

\medskip

\begin{proposition}[Menelaus' Theorem] \label{prop:M} 
Let $ABC$ be a non-degenerate Euclidean triangle and let $A',B',C'$ be three arbitrary   points on the lines 
containing the sides $BC, AC, AB$.   Assume that $A' \neq C$, $B' \neq A$ and 
$C' \neq B$.
Then, the points  $A',B',C'$ are aligned if and only if we have
\[
 \frac{A'B}{A'C} \cdot \frac{B'C}{B'A}  \cdot \frac{C'A}{C'B}   = +1.
\]
\end{proposition}

\begin{tikzpicture}[line cap=round,line join=round,>=triangle 45,x=0.9cm,y=0.9cm]
\clip(0.65,-1.32) rectangle (9.72,5.95);
\draw [domain=0.65:9.72] plot(\x,{(--14.91-2.4*\x)/2.8});
\draw [domain=0.65:9.72] plot(\x,{(-3.35--1.96*\x)/3.21});
\draw [domain=0.65:9.72] plot(\x,{(--24.73-0.44*\x)/6});
\draw [dash pattern=on 2pt off 2pt,domain=0.65:9.72] plot(\x,{(-4.99--1.82*\x)/0.36});
\begin{scriptsize}
\fill [color=black] (1.54,4.01) circle (1.5pt);
\draw[color=black] (1.66,4.2) node {$A$};
\fill [color=black] (4.34,1.61) circle (1.5pt);
\draw[color=black] (4.44,1.81) node {$B$};
\fill [color=black] (7.55,3.57) circle (1.5pt);
\draw[color=black] (7.66,3.77) node {$C$};
\fill [color=black] (2.89,0.72) circle (1.5pt);
\draw[color=black] (3.07,0.67) node {$A'$};
\fill [color=black] (3.25,2.54) circle (1.5pt);
\draw[color=black] (3.55,2.6) node {$C'$};

\fill [color=black] (3.52,3.86) circle (1.5pt);
\draw[color=black] (3.76,4.1) node {$B'$};
\end{scriptsize}
\end{tikzpicture}

\begin{proof} 
Although a purely geometric proof is possible, it is somewhat delicate to correctly handle the signs of the division
ratios throughout the arguments. We follow below a more algebraic approach. 
It will be convenient to assume that $A,B$ and $C$ are points in
$\r^n$ with $n \geq 3$ and to assume that the origin $0\in \r^n$ does not belong to the plane $\pi$ containing
the three points $A,B,C$.  By hypothesis, the point $C'$ lies on the line through $A$ and $B$, therefore
$$
 C' = \nu A + (1-\nu) B, \quad \text{with}  \quad  \frac{C'A}{C'B} = \frac{\nu-1}{\nu}.
$$   
Likewise, we have $A' = \lambda B + (1-\lambda) C$ and $ B' = \mu C + (1-\mu) A$ with  $\frac{A'B}{A'C} = \frac{\lambda-1}{\lambda}$ and   
$\frac{B'C}{B'A} = \frac{\mu-1}{\mu}$.
 Now the point $C'$ lies on the line through $A'$ and $B'$ if and only if there exists $\rho\in \r$ such that 
$C' = \rho A' + (1-\rho)B'$. We then have 
 \begin{align*}
 C' &= \rho (\lambda B + (1-\lambda) C)  +  (1-\rho) ( \mu C + (1-\mu) A) 
 \\  &= (1-\rho)(1-\mu)  A  + \rho \lambda B + (\rho (1-\lambda) + \mu(1-\rho) ) C 
 \\   &=  \nu A + (1-\nu) B.
\end{align*}
 Our hypothesis implies that $A,B,C$ correspond to three linearly independent vectors in $\r^n$, therefore
 the latter identity implies
 $$
  \nu =  (1-\rho)(1-\mu), \quad (1-\nu) = \rho \lambda  \quad  \text{and }   (\rho (1-\lambda) + \mu(1-\rho)) = 0.
 $$ 
 Thus,
  $$
   \frac{(1-\mu)(1-\nu)}{\lambda \nu} =   \frac{\rho}{1-\rho } = -\frac{\mu}{1-\lambda }
 $$
 and we conclude that $C'$ is aligned with $A'$ and $B'$ if and only if
 $$
   \frac{A'B}{A'C} \cdot \frac{B'C}{B'A}  \cdot \frac{C'A}{C'B}  = \frac{(\lambda-1)(\mu-1)(\nu-1)}{\lambda\mu \nu} = 1.
 $$
\end{proof}

\textbf{Remark.}  Using similar arguments, we can also prove Ceva's Theorem. Both theorems are dual to each other. 
Let us recall the statement:\index{Theorem!Ceva}

\begin{proposition}[Ceva's  Theorem] \label{prop:Ceva} 
Let $ABC$ be a non-degenerate Euclidean triangle and let $A',B',C'$ be three arbitrary   points on the lines 
containing the sides $BC, AC, AB$ such  that $A' \neq C$, $B' \neq A$ and 
$C' \neq B$.
Then, the lines $AA'$, $BB'$ and $CC'$ are concurrent or parallel
 if and only if we have
\[
 \frac{A'B}{A'C} \cdot \frac{B'C}{B'A}  \cdot \frac{C'A}{C'B}   = -1.
\]
\end{proposition}

\hspace{3cm} 
\begin{tikzpicture}[line cap=round,line join=round,>=triangle 45,x=0.9cm,y=0.9cm]
\clip(2.8,-0.9) rectangle (9.16,4.92);
\draw [domain=2.8:9.16] plot(\x,{(--12.48-2.57*\x)/0.82});
\draw [domain=2.8:9.16] plot(\x,{(--4.01--0.22*\x)/3.09});
\draw [domain=2.8:9.16] plot(\x,{(--24.63-2.35*\x)/3.91});
\draw [dash pattern=on 1pt off 1pt,domain=2.8:9.16] plot(\x,{(--4.59-1.29*\x)/-2.75});
\draw [dash pattern=on 1pt off 1pt,domain=2.8:9.16] plot(\x,{(--10.44-2.61*\x)/0.3});
\draw [dash pattern=on 1pt off 1pt,domain=2.8:9.16] plot(\x,{(--5.51-1.41*\x)/-0.36});
\begin{scriptsize}
\fill [color=black] (3.52,4.18) circle (1.5pt);
\draw[color=black] (3.75,4.3) node {$A$};
\fill [color=black] (4.34,1.61) circle (1.5pt);
\draw[color=black] (4.5,1.85) node {$B$};
\fill [color=black] (7.43,1.83) circle (1.5pt);
\draw[color=black] (7.5,2.1) node {$C$};
\fill [color=black] (3.82,1.57) circle (1.5pt);
\draw[color=black] (3.55,1.4) node {$A'$};
\fill [color=black] (4.68,0.52) circle (1.5pt);
\draw[color=black] (5.0,0.52) node {$C'$};
\fill [color=black] (3.98,0.2) circle (1.5pt);
\draw[color=black] (3.75,0.221) node {$P$};

\fill [color=black] (4.81,3.41) circle (1.5pt);
\draw[color=black] (5.1,3.53) node {$B'$};
\end{scriptsize}
\end{tikzpicture}
 
\begin{proof}
Let us give the main step of the  proof of Ceva's theorem. Suppose that  the lines $AA'$, $BB'$ and $CC'$ 
meet at a point $P$. Then we can find $r,s$ and $t$ in $\r$ such that
$$
 P = tA + (1-t)A' = sB + (1-s)B' = rC + (1-r)C'.
$$
We also have as before $A' = \lambda B + (1-\lambda) C$,  $B' = \mu C + (1-\mu) A$ 
and $C' = \nu A + (1-\nu) B$. This implies
\begin{eqnarray*}
P & = & tA + (1-t)\lambda B + (1-t)(1-\lambda) C 
 \\ & = & (1-s)(1-\mu) A + sB + (1-s)\mu C
 \\ & = & (1-r) \nu A +(1-r)(1-\nu)B + r C.
\end{eqnarray*}
By uniqueness of the barycentric coordinates with respect to the triangle $ABC$, we have 
\begin{eqnarray*}
t  & = &  (1-s)(1- \mu) =  (1-r) \nu 
 \\ 
 s & = & (1-t)\lambda = (1-r)(1-\nu)
 \\ 
 r & = &  (1-t)(1-\lambda) = (1-s) \mu.
\end{eqnarray*}
Therefore, we have
 $$
   \frac{A'B}{A'C} \cdot \frac{B'C}{B'A}  \cdot \frac{C'A}{C'B}  = \frac{(\lambda-1)(\mu-1)(\nu-1)}{\lambda\mu \nu} 
   = - \frac{r}{s} \cdot \frac{t}{r}\cdot \frac{s}{t}
   = -1.
 $$
 We leave it to the reader to discuss the case where the lines $AA'$, $BB'$ and $CC'$  are parallel and to
 prove the converse direction.
  \end{proof}

\section{The classical proof of the triangle inequality for the Funk metric}

The triangle inequality is proved in Section \ref{sec.triangle}. Note that it also easily follows from Corollary  \ref{cor.dFh2}
(see also  \cite{Yamada2014}), 
and it is also a consequence of the Finslerian description of the Funk metric (see Chapter 3 \cite{TroyanovFinslerian} in this volume).

In this appendix, we present the classical proof of the triangle following Zaustinsky \cite{Zaustinsky}.
This proof is similar to the original proof of the triangle inequality for the Hilbert distance, as given by D. Hilbert in \cite{Hilbert2},
although the proof in the case of the Hilbert distance is a bit simpler and does not use the Menelaus theorem.

We now prove  the triangle inequality for the Funk metric following  \cite[p. 85]{Zaustinsky}.  
Let $x,y,z$ be three points in $\Omega$. In view of Property (b) in Proposition \ref{prop.basicproperties}, we  may assume that they are not collinear.
Let $a,b,c,d,e,f$ be the intersections with $\partial \Omega$ of the lines $xz$, $yx$ and $zy$, using the notation 
of the figure concerning the order of intersections.

\begin{center}
 \definecolor{yqyqyq}{rgb}{0.5,0.5,0.5}
\begin{tikzpicture}[x=0.7cm,y=0.7cm]
\clip(-3.62,-4.69) rectangle (5.04,7.97);
\fill[color=yqyqyq,fill=yqyqyq,fill opacity=0.1] (0,-1.26) -- (2.1,-0.47) -- (0.66,0.67) -- cycle;
  \draw plot  [smooth cycle]  coordinates  {(-1.22,2.16)   (1.34,2.46)   (4.24,0.34)   (3.66,-1.7)   (-0.76,-3.5)     (-2.45,-2.19)}; 
\draw (-0.76,-3.5)-- (1.27,2.45);
\draw (-1.22,2.16)-- (3.66,-1.7);
\draw (-2.45,-2.19)-- (4.24,0.34);
\draw [dash pattern=on 2pt off 2pt] (-1.66,7.53)-- (3.66,-1.7);
\draw [dash pattern=on 2pt off 2pt] (-1.66,7.53)-- (1.17,-0.82);
\draw [dash pattern=on 2pt off 2pt] (-1.66,7.53)-- (-0.76,-3.5);
\draw (0,-1.26)-- (2.1,-0.47);
\draw (2.1,-0.47)-- (0.66,0.67);
\draw  (0.66,0.67)-- (0,-1.26);
\begin{scriptsize}
\fill   (-0.76,-3.5) circle (1.5pt);
\draw  (-0.63,-3.8) node {$b$};
\fill   (-1.22,2.16) circle (1.5pt);
\draw  (-1.5,2.31) node {$d$};
\fill  (1.27,2.45) circle (1.5pt);
\draw (1.38,2.71) node {$a$};
\fill   (3.66,-1.7) circle (1.5pt);
\draw  (3.82,-1.81) node {$c$};
\fill   (-1.66,7.53) circle (1.5pt);
\draw  (-1.52,7.79) node {$p$};
\fill   (0.66,0.67) circle (1.5pt);
\draw  (0.36,0.6) node {$y$};
\fill  (-2.45,-2.19) circle (1.5pt);
\draw (-2.66,-2.2) node {$f$};
\fill   (4.24,0.34) circle (1.5pt);
\draw  (4.41,0.52) node {$e$};
\fill   (1.17,-0.82) circle (1.5pt);
\draw  (1.32,-1.1) node {$y'$};
\fill   (0,-1.26) circle (1.5pt);
\draw  (0.16,-1.4) node {$x$};
\fill  (2.1,-0.47) circle (1.5pt);
\draw  (2.05,-0.7) node {$z$};
\fill  (2.8,-0.2) circle (1.5pt);
\draw  (3,0.1) node {$a'$};
\fill  (-0.91,-1.61) circle (1.5pt);
\draw  (-1.16,-1.5) node {$b'$};
\end{scriptsize}
\end{tikzpicture}
\end{center}  
\vspace{-0.3cm}
From the invariance of the cross ratio from the perspective at $p$, we have 
\[
\frac{ \vert  x-a \vert }{  \vert y-a\vert } \cdot \frac{  \vert b-y\vert }{  \vert b-x\vert } =
\frac{ \vert  x-a'\vert }{  \vert y' -a'\vert } \cdot \frac{  \vert b'-y' \vert }{\vert b'-x\vert }
\]
and 
\[
 \frac{ \vert  y-c\vert }{ \vert  z-c\vert } \cdot \frac{  \vert d-z\vert }{\vert d-y\vert } = 
 \frac{\vert   y' -a'\vert }{  \vert z-a'\vert } \cdot \frac{  \vert b'-z\vert }{ \vert  b'-y' \vert }.
\]
 
Multiplying both sides of these two equations, we get
$$
\frac{ \vert  x-a \vert }{  \vert y-a\vert }\cdot \frac{ \vert  y-c\vert }{ \vert  z-c\vert } = 
\frac{ \vert  x-a'\vert }{  \vert z -a'\vert } \cdot \frac{  \vert b'-z \vert }{\vert b'-x\vert } \cdot
\frac{  \vert b-x\vert } {  \vert b-y\vert } \cdot \frac{\vert d-y\vert } {\vert d-z\vert }.
$$
The three points $b,b'$ and $d$ lie on the sides of the triangle $xyz$ and are aligned,
therefore we have by Menelaus' theorem (Theorem \ref{prop:M}):
$$
\frac{  \vert b'-z \vert }{\vert b'-x\vert } \cdot
\frac{  \vert b-x\vert } {  \vert b-y\vert } \cdot \frac{\vert d-y\vert } {\vert d-z\vert }
= 1
$$
\[
\frac{  \vert d-x\vert }{ \vert  d-y\vert } \cdot \frac{ \vert  f-y\vert }{\vert   f-z\vert } =\frac{\vert   a'-x\vert }{ \vert  a'-z\vert }.
\]
This gives
\[\frac{ \vert  x-a\vert }{ \vert  a-c\vert } \cdot \frac{ \vert  y-c\vert }{  \vert z-c\vert } =\frac{  \vert x-a'\vert }{  \vert z-a'\vert }\geq \frac{ \vert  x-a\vert }{ \vert  z-a\vert },\]
and the inequality is strict unless $a=a'$. This inequality is equivalent to the triangle inequality for the Funk metric\footnote{Observe that the argument shows that   the inequality is strict for all $x,y,z$ unless $\partial \Omega$ contains a Euclidean segment; compare with Theorem
\ref{th.triangle_ineq}.}.
\qed


\end{document}